\documentclass[11pt,a4paper]{article}

\usepackage{amsmath, amssymb, amsthm}
\usepackage{enumerate,enumitem}
\setenumerate{leftmargin=*}
\usepackage{color} 
\newcommand\norm[1]{\left\lVert#1\right\rVert}

\newtheorem{thm}{Theorem}[section]   
\newtheorem{Remark}[thm]{Remark}

\newtheorem{Concluding Remarks}[thm]{Concluding Remarks}   
\newtheorem{Lemma}[thm]{Lemma} 
\newtheorem{exampl}[thm]{Example}  
\newtheorem{Proposition}[thm]{Proposition}   
\newtheorem{Corollary}[thm]{Corollary}

\newtheorem{Definition}[thm]{Definition} 
\newcommand{\N} {{\mathbb N}}

\newtheorem{mytheorem}{Theorem}[subsection]

\newtheorem{myCorollary}[mytheorem]{Corollary}

\title{Lipschitz $\left(\mathfrak{m}^L\left(s;q\right),p\right)$ and $\left(p,\mathfrak{m}^L\left(s;q\right)\right)-$summing maps}
\date{\today}
\author{MANAF ADNAN SALAH \\ Mathematisches Institut, Universit\"at Jena \\ Ernst$-$Abbe$-$Platz 2, 07743 Jena, Germany\\ E$-$mails: manaf$-$adnan.salah@uni$-$jena.de \\ \ \ \:manaf$_{-}$math@yahoo.com}
\begin{document}
\maketitle

\begin{abstract}
Building upon the linear version of mixed summable sequences in arbitrary Banach spaces of A. Pietsch, we introduce a nonlinear version of his concept and study its properties. Extending previous work of J. D. Farmer, W. B. Johnson and J. A. Ch{\'a}vez-Dom{\'{\i}}nguez, we define Lipschitz $\left(\mathfrak{m}^L\left(s;q\right),p\right)$ and Lipschitz $\left(p,\mathfrak{m}^L\left(s;q\right)\right)-$summing maps and establish inclusion theorems, composition theorems and several characterizations. Furthermore, we prove that the classes of Lipschitz $\left(r,\mathfrak{m}^L\left(r;r\right)\right)-$summing maps with $0<r<1$ coincide. We obtain that every Lipschitz map is Lipschitz $\left(p,\mathfrak{m}^L\left(s;q\right)\right)-$summing map with $1\leq s< p$ and $0<q\leq s$ and discuss a sufficient condition for a Lipschitz composition formula as in the linear case of A. Pietsch. Moreover, we discuss a counterexample of the nonlinear composition formula, thus solving a problem by J. D. Farmer and W. B. Johnson.
\end{abstract}

\section{INTRODUCTION}
The starting point of the theory of linear operator ideals is the fundamental work of A. Grothendieck \cite{Gro96} together with the study of operator ideals in Hilbert spaces by I. Gohberg, M. G. Krejn \cite{GoKre69} and R. Schatten \cite{Sch60}. This work was done in the middle of the last century. 

The book \cite{P80} by A. Pietsch constitutes a culminating point in the development. Two things are established there. First, a general theory of operator ideals is presented. Second, a wealth of important examples of operator ideals are treated in detail. This lead to widespread applications not only in Banach space theory and operator theory, but also in harmonic analysis and approximation theory.

One particular important class of operator ideals is the class of $p-$summing operators considered already by A. Grothendieck in the cases $p=1$ and $p=2$ and generalized by A. Pietsch to $1\le p <\infty$. 

A. Pietsch established many of their fundamental properties. The celebrated Pietsch Domination and Factorization Theorem is proved
here. A beautiful consequence of the Pietsch Domination and Factorization Theorems is the composition theorem
for $p-$summing operators. 

A. Pietsch \cite{P80} and Maurey \cite{M74} have studied the characterizations of $(s;q)-$mixing operators. M{\'a}rio C. Matos \cite{M04},\cite{M07} first studied the concepts $(p,\mathfrak{m}(s;q))$ and $\left(\mathfrak{m}\left(s;q\right),p\right)-$summing mappings with $0<q\leq s\leq\infty$; $p\geq q$ and $p\leq q$, respectively.

J. D. Farmer and W. B. Johnson \cite{J09} have formally introduced the concept of a Lipschitz $p-$summing operator 
between metric spaces with $1\leq p<\infty$, although this notion already played a role in earlier work of J. Bourgain in \cite{B85}. 

The paper of J. Bourgain found applications in computer science, so the Lipschitz $p-$summing operators are expected
to play a similar important role for applications in the nonlinear case as the linear $p-$summing operators for
the linear theory.

J. D. Farmer and W. B. Johnson proved that this is a true extension of the linear concept and  obtained a nonlinear counterpart of important fundamental characterizations of $p-$summing linear operators.

This was done by showing a nonlinear version of the Pietsch Domination and Factorization Theorem. 
With this theorem they proved that the Lipschitz $p-$summing norm of a linear operator is the same as its $p-$summing norm. 

J. A. Ch{\'a}vez-Dom{\'{\i}}nguez \cite{JA12} introduced the nonlinear concept of Lipschitz $(s;q)-$mixing operators and proved several characterizations.

In the present paper, the corresponding concepts of Lipschitz mixed $(s;q)-$summable sequences, Lipschitz $\left(\mathfrak{m}^L\left(s;q\right),p\right)$
and Lipschitz $\left(p,\mathfrak{m}^L\left(s;q\right)\right)-$summing maps are defined and studied, respectively. 

We start by recalling the definitions of mixed summable sequences in arbitrary Banach spaces and various special cases of $\left(p,\mathfrak{m}\left(s;q\right)\right)$ and $\left(\mathfrak{m}\left(s;q\right),p\right)-$summing bounded linear operators in arbitrary Banach spaces. Then we present fundamental definitions and basic properties of Lipschitz mixed $(s;q)-$summable sequences, Lipschitz $\left(p,\mathfrak{m}^L\left(s;q\right)\right)$ and Lipschitz $\left(\mathfrak{m}^L\left(s;q\right),p\right)-$summing maps. Afterwards several characterizations and useful results such as inclusion theorems and composition theorems are established. 

Furthermore, we prove that the classes of Lipschitz $\left(r,\mathfrak{m}^L\left(r;r\right)\right)-$summing maps with $0<r<1$ coincide. We obtain that every Lipschitz map is Lipschitz $\left(p,\mathfrak{m}^L\left(s;q\right)\right)-$summing with $1\leq s< p$ and $0<q\leq s$ and discuss a sufficient condition for a Lipschitz composition formula as in the linear case of A. Pietsch \cite{P67}. 

Moreover, we discuss a counterexample of the nonlinear composition formula, thus solving a problem by J. D. Farmer and W. B. Johnson \cite[Problem 1]{J09}.

\section{NOTATIONS AND PRELIMINARIES}
We introduce concepts and notations that will be used in this article. The letters $X$, $Y$ and $Z$ will denote pointed metric spaces, i.e. each one has a special point designated by $x_{0}$, $y_{0}$ and $z_{0}$, respectively. The letters $E$, $F$ and $G$ will denote Banach spaces. The closed unit ball of a Banach space $E$ is denoted by $B_{E}$. The dual space of $E$ is $E^{*}$. The class of all bounded linear operators between arbitrary Banach spaces will be denoted by $\mathfrak{L}$. The symbols $\mathbb{R}$ and $\mathbb{N}$ stand for the set of all real numbers and the set of all natural numbers, respectively. For the Lipschitz mapping $T$ between metric spaces, $Lip(T)$ denotes its Lipschitz constant. 

Given metric spaces $X$ and $Y$, the set of all Lipschitz functions from $X$ into $Y$ that send the special point $x_{0}$ to $y_{0}$ will be denoted by $\mathbb{L}_{x_{0}}(X,Y)$ and the set of all Lipschitz functions from $X$ into $Y$ will be denoted by $\mathbb{L}(X,Y)$. For the special case $Y=\mathbb{R}$, the Banach space of real$-$valued Lipschitz functions defined on $X$ that send the special point $x_{0}$ to $0$ with the Lipschitz norm $Lip(\cdot)$ will be denoted by $X^{\#}$. The space $X^{\#}$ is called Lipschitz dual of $X$. The symbol $W(B_{X^{\#}})$ stands for the set of all Borel probability measures defined on $B_{X^{\#}}$.

In this paper, we write $\kappa=(\kappa_j)_{j\in\mathbb{N}}\subset\mathbb{R^{+}}$; $\eta=(\eta_j)_{j\in\mathbb{N}}\subset\mathbb{R}$; $\zeta=(\zeta_j)_{j\in\mathbb{N}}\subset\mathbb{R}$; $\alpha=(\alpha_j)_{j\in\mathbb{N}}\subset\mathbb{R}$; $\lambda=(\lambda_j)_{j\in\mathbb{N}}\subset\mathbb{R}\backslash\left\{0\right\}$; $\sigma=(\sigma_j)_{j\in\mathbb{N}}\subset\mathbb{R}\backslash\left\{0\right\}$ and $x'=(x'_j)_{j\in\mathbb{N}}$; $x''=(x''_j)_{j\in\mathbb{N}}$ are abbreviations for the corresponding sequences in $X$. 

In contrast to the situation in Banach spaces $E$, where it is enough to consider sequences $(x_{j})_{j\in\mathbb{N}}\subset E$, we need to consider sequences $\big((\sigma_j,x'_j,x''_j)\big)_{j\in\mathbb{N}}$ of triples
$(\sigma_j,x'_j,x''_j)\in\mathbb{R}\times X\times X$. To simplify notation we write $(\sigma,x',x'')=\big((\sigma_j,x'_j,x''_j)\big)_{j\in\mathbb{N}}\subset\mathbb{R}\times X\times X$, for such a sequence.

If $\tau=(\tau_j)_{j\in\mathbb{N}}\subset\mathbb{R}\backslash\left\{0\right\}$ is a scalar sequence, then we write $(\frac{\sigma}{\tau},x',x'')=\big((\frac{\sigma_j}{\tau_j},x'_j,x''_j)\big)_{j\in\mathbb{N}}$.\\
Let $0<p<\infty$. The $p-$sequence set, denoted by $\ell_p(\mathbb{N},\mathbb{R}\times X\times X)$ or $\ell_p(\mathbb{R}\times X\times X)$, is defined as $$\ell_p(\mathbb{R}\times X\times X)=\left\{(\sigma,x',x'')\subset\mathbb{R}\times X\times X:\sum\limits_{j=1}^{\rm\infty}\left|\sigma_j\right|^{p} d_X(x'_j,x''_j)^{p}<\infty\right\}.$$
We denote its strong $p-$norm by $$\left\|(\sigma,x',x'')\Big|\ell_p\right\|=\Bigg[\sum\limits_{j=1}^{\rm\infty}\left|\sigma_j\right|^{p} d_X(x'_j,x''_j)^{p}\Bigg]^{\frac{1}{p}}.$$ Also the weak Lipschitz $p-$sequence set, denoted $\ell_p^{L,w}(\mathbb{N},\mathbb{R}\times X\times X)$ or $\ell_p^{L,w}(\mathbb{R}\times X\times X)$, is defined as $$\ell_p^{L,w}(\mathbb{R}\times X\times X)=\left\{(\sigma,x',x'')\subset\mathbb{R}\times X\times X:\sup\limits_{f\in B_{{X}^{\#}}}\sum\limits_{j=1}^{\infty}\left|\sigma_j\right|^{p}\left|fx'_j-fx''_j\right|^{p}<\infty\right\}.$$
We denote its weak Lipschitz $p-$norm by $$\left\|(\sigma,x',x'')\Big|\ell_p^{L,w}\right\|=\sup\limits_{f\in B_{{X}^{\#}}}\Bigg[\sum\limits_{j=1}^{\infty}\left|\sigma_j\right|^{p}\left|fx'_j-fx''_j\right|^{p}\Bigg]^{\frac{1}{p}}.$$
In the case $p=\infty$, the $\infty -$sequence set, denoted by $\ell_\infty(\mathbb{N},\mathbb{R}\times X\times X)$ or $\ell_\infty(\mathbb{R}\times X\times X)$, is defined as $$\ell_\infty(\mathbb{R}\times X\times X)=\left\{(\sigma,x',x'')\subset\mathbb{R}\times X\times X:\left|\sigma_j\right| d_X(x'_j,x''_j)<\infty\right\}.$$
We denote its $\infty-$norm by $$\left\|(\sigma,x',x'')\Big|\ell_\infty\right\|=\sup\limits_{j\in\mathbb{N}}\left|\sigma_j\right| d_X(x'_j,x''_j).$$ 
Also we denote its weak Lipschitz $\infty-$norm by $$\left\|(\sigma,x',x'')\Big|\ell_\infty^{L,w}\right\|=\sup\limits_{f\in B_{{X}^{\#}}}\sup\limits_{j\in\mathbb{N}}\left|\sigma_j\right|\left|fx'_j-fx''_j\right|.$$ 
It is obvious that $$\left\|(\sigma,x',x'')\Big|\ell_\infty\right\|=\left\|(\sigma,x',x'')\Big|\ell_\infty^{L,w}\right\|.$$
The same notations are used for finite sequences of the same length.
 
 Observe that, since there is no linear structure on the set of triples $(\sigma,x',x'')$, the above notions are not really norms. But because of the similarity with the usual $\ell_{p}-$norm, we shall call them norms.
 
 Let $0<q\leq s\leq\infty$ and let the index $s'\left(q\right)$ is determined by the equation  $$\frac{1}{s'\left(q\right)}+\frac{1}{s}=\frac{1}{q}.$$
In this case we say that $s$ and $s'\left(q\right)$ are $q-$conjugate. We also denote $s'\left(1\right)$ by $s'$. In this case $s$ and $s'$ are conjugate in the usual sense. 
 
 Recall that, for $1\leq p<\infty$, a bounded linear operator $T$ from $E$ into $F$ is called $p-$summing if there is a nonnegative constant $C_1$ such that for all $m\in\mathbb{N}$ and any vectors $x_j\in E$, the inequality 
\begin{equation}\label{20}
\sum\limits_{j=1}^{m}\left\|Tx_j\right\|^p\leq C_{1}^{p}\cdot\sup\limits_{x^{*}\in B_{{E}^{*}}}\sum\limits_{j=1}^{m}\left|x^{*}(x_j)\right|^{p}
\end{equation} holds. In this case, the $p-$summing norm $\pi_{p}(T)$ of $T$ is the infimum of such constants $C_1$.
 
 Inspired by this useful concept, J. D. Farmer and W. B. Johnson \cite{J09} defined the Lipschitz $p-$\\summing norm $\pi_{p}^{L}(T)$ of a (not necessarily linear) mapping $T$ from $X$ into $Y$ as the infimum of all nonnegative constants $C_2$ such that for all $m\in\mathbb{N}$, any sequences $x'$, $x''$ in $X$ and $\kappa$ in $R^{+}$, the inequality 
$$\left\|(\kappa,Tx',Tx'')\Big|\ell_p\right\|\leq C_{2}\cdot\left\|(\kappa,x',x'')\Big|\ell_p^{L,w}\right\|$$ holds. This definition remains unchanged if we consider only the case $\kappa_j=1$, a very useful observation in \cite{J09} also credited to M. Mendel and G. Schechtman. The set of all Lipschitz $p-$summing maps from $X$ to $Y$ is denoted by $\Pi_{p}^{L}(X,Y)$.
 
 Recall that the definition of mixed summable sequences in arbitrary Banach spaces of A. Pietsch \cite{P80} is as follows. Let $0<q\leq s\leq\infty$ and let the index $s'\left(q\right)$ is determined by the equation  $$\frac{1}{s'\left(q\right)}+\frac{1}{s}=\frac{1}{q}.$$ A sequence $x\subset E$, is called mixed $(s;q)-$summable, if there exists a sequence $\zeta\in\ell_{s'\left(q\right)}$ and a sequence $x^{0}\in\ell_s^{w}(E)$ such that $x_j=\zeta_{j}\cdot x^{0}_{j}$, $\forall\ j\in\mathbb{N}$.

We denote by $\ell_{(s;q)}^{m}(E)$ the vector space of all mixed $(s;q)-$summable sequences of elements of $E$. For $x\in\ell_{(s;q)}^{m}(E)$ we set 
\begin{equation}\label{zero}
\left\|x\Big|\ell_{(s;q)}^{m}(E)\right\|=\inf\left\|\zeta\Big|\ell_{s'\left(q\right)}\right\|\left\|x^{0}\Big|\ell_s^{w}(E)\right\|,
\end{equation}
where the infimum is taken over all possible factorizations.\\
On $\ell_{(s;q)}^{m}(E)$, $\left\|(\cdot)\Big|\ell_{(s;q)}^{m}(E)\right\|$ defined by (\ref{zero}) is a norm for $q\geq 1$ and a $q-$norm for $0<q<1$. \\
We abbreviate $\left\|x\Big|\ell_{(s;s)}^{m}(E)\right\|=\left\|x\Big|\ell_s^{w}(E)\right\|$, for every sequence $x\subset E$.

 For $0<q\leq s\leq\infty$ and $p\geq q$, M{\'a}rio C. Matos \cite{M04} defined the $\left(p,\mathfrak{m}\left(s;q\right)\right)-$summing norm $\left\|T\right\|_{\left(p,\mathfrak{m}\left(s;q\right)\right)}$ of a bounded linear mapping $T$ from $E$ into $F$ as the infimum of all nonnegative constants $C_3$ such that for all $m\in\mathbb{N}$ and any vectors $x_j\in E$ the inequality 
\begin{equation}\label{q1}\sum\limits_{j=1}^{m}\left\|Tx_j\right\|^p\leq C_{3}^{p}\cdot \left\|(x_j)_{j}\right\|_{\mathfrak{m}\left(s;q\right)}^{p},
\end{equation}
holds. The vector space of all $\left(p,\mathfrak{m}\left(s;q\right)\right)-$summing bounded linear mappings from $E$ into $F$ is denoted by $\mathfrak{L}_{\left(p,\mathfrak{m}\left(s;q\right)\right)}(E,F)$.

 The following special cases of $\left(p,\mathfrak{m}\left(s;q\right)\right)-$summing bounded linear maps were already considered by A. Pietsch \cite{P80}. If $s=q$, then the $\left(p,\mathfrak{m}\left(s;s\right)\right)-$summing norm $\left\|T\right\|_{\left(p,\mathfrak{m}\left(s;s\right)\right)}$ of the mapping $T$ is the infimum of all nonnegative constants $C_4$ such that for all $m\in\mathbb{N}$ and any vectors $x_j\in E$ the inequality 
$$\left[\sum\limits_{j=1}^{m}\left\|Tx_j\right\|^p\right]^{\frac{1}{p}}\leq C_{4}\cdot\sup\limits_{x^{*}\in B_{{E}^{*}}}\left[\sum\limits_{j=1}^{m}\left|x^{*}(x_j)\right|^{s}\right]^{\frac{1}{s}}$$
holds. This is the usual $(p,s)-$summing norm of $T$.

If $s=q=p$, then the $\left(p,\mathfrak{m}\left(p;p\right)\right)-$summing norm $\left\|T\right\|_{\left(p,\mathfrak{m}\left(p;p\right)\right)}$ of the mapping $T$ is the infimum of all nonnegative constants $C_5$ such that for all $m\in\mathbb{N}$ and any vectors $x_j\in E$ the inequality (\ref{20}) holds. So we obtain the usual $p-$summing norm of $T$.

For $0<q\leq s\leq\infty$ and $p\leq q$, M{\'a}rio C. Matos \cite{M07} also defined the $\left(\mathfrak{m}\left(s;q\right),p\right)-$summing norm $\left\|T\right\|_{\left(\mathfrak{m}\left(s;q\right),p\right)}$of a bounded linear mapping $T$ from $E$ into $F$ as the infimum of all nonnegative constants $C_6$ such that for all $m\in\mathbb{N}$ and any vectors $x_j\in E$  the inequality
\begin{equation}\label{q2} 
\left\|(Tx_j)_j\right\|^{p}_{\mathfrak{m}\left(s;q\right)}\leq C_{6}^{p}\cdot\sup\limits_{x^{*}\in B_{{E}^{*}}}\sum\limits_{j=1}^{m}\left|x^{*}(x_j)\right|^{p}
\end{equation} holds. The vector space of all $\left(\mathfrak{m}\left(s;q\right),p\right)-$summing bounded linear mappings from $E$ into $F$ is denoted by $\mathfrak{L}_{\left(\mathfrak{m}\left(s;q\right),p\right)}(E,F)$. 

The following special case of $\left(\mathfrak{m}\left(s;q\right),p\right)-$summing bounded linear maps were already considered by A. Pietsch \cite{P80}. If $p=q$, then the $\left(\mathfrak{m}\left(s;p\right),p\right)-$summing norm $\left\|T\right\|_{\left(\mathfrak{m}\left(s;p\right),p\right)}$ of the mapping $T$ is the infimum of all nonnegative constants $C_7$ such that for all $m\in\mathbb{N}$ and any vectors $x_j\in E$ the inequality $$\left\|(Tx_j)_j\right\|^{p}_{\mathfrak{m}\left(s;p\right)}\leq C_{7}^{p}\cdot\sup\limits_{x^{*}\in B_{{E}^{*}}}\sum\limits_{j=1}^{m}\left|x^{*}(x_j)\right|^{p}$$ holds.

\begin{Remark}
M{\'a}rio C. Matos \cite{M07} proved if a map $T$ from $E$ into $F$ satisfying the inequality (\ref{q1}) with $p<q$ and the inequality (\ref{q2}) with $p>q$, respectively, then in both cases $T$ is a zero map.
\end{Remark}

A. Pietsch \cite[Chap. 21]{P80} defined the ideal of operators possessing $(s,p)-$type for $0<p<s\leq 2$. The theory of these operators was created by B. Maurey \cite{M72}. 

For every finite sequence $x\subset E$, we put 
\begin{equation}\label{type2}
t_{(s,p)}\left(x_k\right)=c_{sp}^{-1}\cdot\Bigg(\int\limits_{\mathbb{R}^{n}}\left\|\sum\limits_{k=1}^{n}t_{k}\cdot x_{k}\right\|^{p} d\mu_{s}^{n}(t)\Bigg)^{\frac{1}{p}}.
\end{equation}
Here  $t=(t_{1},\cdots,t_{n})\in\mathbb{R}^{n}$ and $\mu_{s}^{n}$ denotes the $n-$fold product of $s-$stable laws $\mu_{s}$ were invented by P. L{\'e}vy \cite{P23}.

An operator $S\in\mathfrak{L}(E,F)$ is said to be of $(s,p)-$type if there exists a constant $\varrho\geq 0$ such that
\begin{equation}\label{type3}
t_{(s,p)}\left(Sx\right)\leq\varrho\left\|x|\ell_s(E)\right\|
\end{equation} 
for arbitrary finite sequence $x$ in $E$; $k=1,\cdots,n$ and $n\in\mathbb{N}$. We put $\mathbf{T}_{(s,p)}(S)=\inf\varrho$. 

 The class of these operators is denoted by $\mathfrak{T}_{(s,p)}$. For further reference, we recall the following theorem.
\begin{thm}\cite[Sec. 21]{P80} \label{type4}
If $0<p<s<1$. Then $\mathfrak{T}_{(s,p)}=\mathfrak{L}$.
\end{thm}

The absolute moments $$c_{sp}=\left(\int\limits_{\mathbb{R}}\left|l\right|^{p}d\mu_{s}(l)\right)^{\frac{1}{p}}=2\cdot\Bigg[\frac{\Gamma\left(\frac{s-p}{s}\right)\cdot\Gamma\left(\frac{1+p}{2}\right)}{\Gamma\left(\frac{2-p}{2}\right)\cdot\Gamma\left(\frac{1}{2}\right)}\Bigg]^{\frac{1}{p}}$$ exist for $0<p<s<2$. We also have 
 
\begin{equation}\label{type1}
\Bigg(\int\limits_{\mathbb{R}^{n}}\left|\sum\limits_{k=1}^{n}t_{k}\cdot\xi_{k}\right|^{p} d\mu_{s}^{n}(t)\Bigg)^{\frac{1}{p}}=c_{sp}\cdot\Bigg[\sum\limits_{k=1}^{n}\left|\xi_{k}\right|^{s}\Bigg]^{\frac{1}{s}}
\end{equation}
for $\xi_{1},\cdots,\xi_{n}\in\mathbb{R}$ and $n\in\mathbb{N}$.

\section{DEFINITIONS AND ELEMENTARY PROPERTIES}
In this section, the concepts of Lipschitz mixed $(s;q)-$summable sequences, Lipschitz $\left(\mathfrak{m}^L\left(s;q\right),p\right)$ and Lipschitz $\left(p,\mathfrak{m}^L\left(s;q\right)\right)-$summing maps are defined and studied. Several properties, characterizations and remarks relevant of them help us to establish further results in the next sections. 

Inspired by the definition of mixed summable sequences in arbitrary Banach spaces of A. Pietsch \cite{P80}, we present the following definition.  
\begin{Definition}\label{car}
Let $0<q\leq s\leq\infty$ and let the index $s'\left(q\right)$ is determined by the equation  $$\frac{1}{s'\left(q\right)}+\frac{1}{s}=\frac{1}{q}.$$ A sequence $(\sigma,x',x'')\subset\mathbb{R}\times X\times X$ is called Lipschitz mixed $(s;q)-$summable, if there exists a sequence $\tau\in \ell_{s'\left(q\right)}$ such that $(\frac{\sigma}{\tau},x',x'')\in\ell_s^{L,w}(\mathbb{R}\times X\times X)$.

The class of all Lipschitz mixed $(s;q)-$summable sequences is denoted by $\mathfrak{M}_{(s;q)}^{L}(\mathbb{R}\times X\times X)$. Moreover, for a sequence $(\sigma,x',x'')\in\mathfrak{M}_{(s;q)}^{L}(\mathbb{R}\times X\times X)$ define 
\begin{equation}\label{three}
\mathfrak{m}_{(s;q)}^{L}(\sigma,x',x'')=\inf\left\|\tau\Big|\ell_{s'\left(q\right)}\right\|\left\|(\frac{\sigma}{\tau},x',x'')\Big|\ell_s^{L,w}\right\|
\end{equation}
where the infimum is taken over all sequences $\tau\in\ell_{s'\left(q\right)}$.     
\end{Definition}

\begin{Definition}
Let $0<q<\infty$. We denote by $\mathfrak{M}_{(q;q)}^{L,0}(\mathbb{R}\times X\times X)$ the class of all sequence $(\sigma,x',x'')\subset\mathbb{R}\times X\times X$ such that there exists a sequence $\tau\in c_{0}$ with $(\frac{\sigma}{\tau},x',x'')\in\ell_q^{L,w}(\mathbb{R}\times X\times X)$. Moreover, for a sequence $(\sigma,x',x'')\in\mathfrak{M}_{(q;q)}^{L,0}(\mathbb{R}\times X\times X)$ define 
$$\mathfrak{m}_{(q;q)}^{L,0}(\sigma,x',x'')=\inf\left\|\tau\Big|\ell_{\infty}\right\|\left\|(\frac{\sigma}{\tau},x',x'')\Big|\ell_q^{L,w}\right\|$$
where the infimum is taken over all sequences $\tau\in c_{0}$.     
\end{Definition}

 The proof of the next proposition is similar to \cite [Proposition 4.2]{JA12} and is therefore omitted.  

\begin{Proposition} \label{two} 
Let $0<q<s<\infty$. A sequence $(\sigma,x',x'')$ is Lipschitz mixed $(s;q)-$summable if and only if $$\Bigg[\sum\limits_{j=1}^{\infty}\bigg[\int\limits_{B_{X^{\#}}}\left|\sigma_j\right|^{s}\left|f(x'_j)-f(x''_j)\right|^{s}d\mu(f)\bigg]^\frac{q}{s}\Bigg]^{\frac{1}{q}}<\infty,$$ for every $\mu\in W(B_{X^{\#}})$. In this case$$\sup\limits_{\mu\in W(B_{X^{\#}})}\Bigg[\sum\limits_{j=1}^{\infty}\bigg[\int\limits_{B_{X^{\#}}}\left|\sigma_j\right|^{s}\left|f(x'_j)-f(x''_j)\right|^{s}d\mu(f)\bigg]^\frac{q}{s}\Bigg]^{\frac{1}{q}}=\mathfrak{m}_{(s;q)}^{L}(\sigma,x',x'').$$
\end{Proposition}

\begin{Definition}\label{seven}
Let $0<q\leq s\leq\infty$ and $p\leq q$. A Lipschitz map $T$ from $X$ into $Y$ is called Lipschitz $\left(\mathfrak{m}^L\left(s;q\right),p\right)-$summing if there is a constant $C_{8}\geq 0$ such that
\begin{equation}\label{flat1}
\mathfrak{m}_{(s;q)}^{L}(\sigma,Tx',Tx'')\leq C_{8}\cdot\left\|(\sigma,x',x'')\Big|\ell_p^{L,w}\right\|
\end{equation}  
for arbitrary finite sequences $x'$, $x''$ in $X$ and $\sigma$ in $\mathbb{R}$. 

Let us denote by $\Pi_{\left(\mathfrak{m}^L\left(s;q\right),p\right)}^{L}(X,Y)$ the class of all Lipschitz $\left(\mathfrak{m}^L\left(s;q\right),p\right)-$summing maps from $X$ into $Y$ with $$\pi_{\left(\mathfrak{m}^L\left(s;q\right),p\right)}^{L}(T)=\inf C_{8}.$$
\end{Definition}

\begin{Remark}
\begin{enumerate}
\item An equivalent definition of Lipschitz $\left(\mathfrak{m}^L\left(s;q\right),p\right)-$summing map is as follows. A Lipschitz map $T$ from $X$ into $Y$ is  Lipschitz $\left(\mathfrak{m}^L\left(s;q\right),p\right)-$summing, if every Lipschitz weakly $p-$summable sequence is mapped to a Lipschitz mixed $(s;q)-$summable sequence.
\item The linear space $\Pi_{\left(\mathfrak{m}^L\left(s;q\right),p\right)}^{L}(X,E)$ equipped with the norm $\pi_{\left(\mathfrak{m}^L\left(s;q\right),p\right)}^{L}(\cdot)$ is a Banach space if $q\geq 1$ and a complete $q-$normed space if $0<q<1$.
\end{enumerate}
\end{Remark}

\begin{Definition}\label{eight}
Let $0<q\leq s\leq\infty$ and $p\geq q$. A Lipschitz map $T$ from $X$ into $Y$ is called Lipschitz $\left(p,\mathfrak{m}^L\left(s;q\right)\right)-$summing if there is a constant $C_{9}\geq 0$  such that 
\begin{equation}\label{flat2}
\left\|(\sigma,Tx',Tx'')\Big|\ell_p\right\|\leq C_{9}\cdot\mathfrak{m}_{(s;q)}^{L}(\sigma,x',x'')
\end{equation}
for arbitrary finite sequences $x'$, $x''$ in $X$ and $\sigma$ in $\mathbb{R}$. 

Let us denote by $\Pi_{\left(p,\mathfrak{m}^L\left(s;q\right)\right)}^{L}(X,Y)$ the class of all Lipschitz $\left(p,\mathfrak{m}^L\left(s;q\right)\right)-$summing maps from $X$ into $Y$ with $$\pi_{\left(p,\mathfrak{m}^L\left(s;q\right)\right)}^{L}(T)=\inf C_{9}.$$
\end{Definition}

\begin{Remark}
\begin{enumerate}
\item An equivalent definition of Lipschitz $\left(p,\mathfrak{m}^L\left(s;q\right)\right)-$summing map is as follows. A Lipschitz map $T$ from $X$ into $Y$ is Lipschitz $\left(p,\mathfrak{m}^L\left(s;q\right)\right)-$summing, if every Lipschitz mixed $(s;q)-$summable sequence is mapped to a strong $p-$summable sequence. 
\item The linear space $\Pi_{\left(p,\mathfrak{m}^L\left(s;q\right)\right)}^{L}(X,E)$ equipped with the norm $\pi_{\left(p,\mathfrak{m}^L\left(s;q\right)\right)}^{L}(\cdot)$ is a Banach space if $p\geq 1$ and a complete $p-$normed space if $0<p<1$.
\end{enumerate}
\end{Remark}

\begin{Concluding Remarks}\label{66}
\begin{enumerate}
\item 
If we consider a Lipschitz map $T$ from $X$ into $Y$ satisfying the inequality (\ref{flat1}) with $p>q$ and the inequality (\ref{flat2}) with $p<q$, respectively, then in both cases $T$ is a constant map, i.e. $T(x)=y_{0}$ for every $x\in X$. Hence the class of Lipschitz $\left(\mathfrak{m}^L\left(s;q\right),p\right)$\-summing maps is only interesting for $p\leq q$ and the class of Lipschitz $\left(p,\mathfrak{m}^L\left(s;q\right)\right)-$summing maps is only interesting for $p\geq q$. 
\item
Let $(\sigma,x',x'')\in\mathfrak{M}_{(s;q)}^{L}(\mathbb{R}\times X\times X)$.
\begin{itemize}
\item If $q=s$, then $\mathfrak{M}_{(q;q)}^{L}(\mathbb{R}\times X\times X)=\ell_q^{L,w}(\mathbb{R}\times X\times X)$ with $\mathfrak{m}_{(q;q)}^{L}(\sigma,x',x'')=\left\|(\sigma,x',x'')\Big|\ell_q^{L,w}\right\|$.	
\item If $s=\infty$, then $\mathfrak{M}_{(\infty;q)}^{L}(\mathbb{R}\times X\times X)=\ell_q(\mathbb{R}\times X\times X)$ with $\mathfrak{m}_{(\infty;q)}^{L}(\sigma,x',x'')=\left\|(\sigma,x',x'')\Big|\ell_q\right\|$.
\end{itemize}	
\item It is obvious that the Lipschitz $\left(\mathfrak{m}^L\left(s;q\right),p\right)-$summing maps satisfy the ideal property, i.e. $$\pi_{\left(\mathfrak{m}^L\left(s;q\right),p\right)}^{L}(S\circ T\circ R)\leq Lip(S)\cdot \pi_{\left(\mathfrak{m}^L\left(s;q\right),p\right)}^{L}(T)\cdot Lip(R)$$ whenever the composition makes sense and also the Lipschitz $\left(p,\mathfrak{m}^L\left(s;q\right)\right)-$summing map satisfy the ideal property, i.e. $$\pi_{\left(p,\mathfrak{m}^L\left(s;q\right)\right)}^{L}(S\circ T\circ R)\leq Lip(S)\cdot\pi_{\left(p,\mathfrak{m}^L\left(s;q\right)\right)}^{L}(T)\cdot Lip(R)$$ whenever the composition makes sense.
\item
The following inclusion results are obvious:
\begin{itemize}
\item Let $0<q< s\leq\infty$. Then $$\mathfrak{M}_{(s;q)}^{L}(\mathbb{R}\times X\times X)\subset\mathfrak{M}_{(q;q)}^{L,0}(\mathbb{R}\times X\times X)\subset\mathfrak{M}_{(q;q)}^{L}(\mathbb{R}\times X\times X).$$ Moreover $$\mathfrak{m}_{(q;q)}^{L}(\sigma,x',x'')\leq\mathfrak{m}_{(q;q)}^{L,0}(\sigma,x',x'')\leq\mathfrak{m}_{(s;q)}^{L}(\sigma,x',x'')$$ for every $(\sigma,x',x'')\in\mathfrak{M}_{(s;q)}^{L}(\mathbb{R}\times X\times X)$.
\item Let $0<q\leq s\leq\infty$ and let the index $s'\left(q\right)$ is determined by the equation 
$$\frac{1}{s'\left(q\right)}+\frac{1}{s}=\frac{1}{q}.$$
Then $\mathfrak{M}_{(s;q)}^{L}(\mathbb{R}\times X\times X)\subset\ell_{s'\left(q\right)}(\mathbb{R}\times X\times X)$. Moreover
$$\left\|(\sigma,x',x'')\Big|\ell_{s'\left(q\right)}\right\|\leq\mathfrak{m}_{(s;q)}^{L}(\sigma,x',x'').$$
\item If $0<q\leq s_1\leq s_2\leq\infty$, then $\mathfrak{M}_{(s_2;q)}^{L}(\mathbb{R}\times X\times X)\subset\mathfrak{M}_{(s_1;q)}^{L}(\mathbb{R}\times X\times X)$. Moreover $$\mathfrak{m}_{(s_1;q)}^{L}(\sigma,x',x'')\leq\mathfrak{m}_{(s_2;q)}^{L}(\sigma,x',x'')$$ for every $(\sigma,x',x'')\in\mathfrak{M}_{(s_2;q)}^{L}(\mathbb{R}\times X\times X)$. Thus it follow that $$\Pi_{\left(p,\mathfrak{m}^L\left(s_1;q\right)\right)}^{L}(X,Y)\subset\Pi_{\left(p,\mathfrak{m}^L\left(s_2;q\right)\right)}^{L}(X,Y).$$ Moreover $$\pi_{\left(p,\mathfrak{m}^L\left(s_2;q\right)\right)}^{L}(T)\leq\pi_{\left(p,\mathfrak{m}^L\left(s_1;q\right)\right)}^{L}(T)$$
for every $T\in\Pi_{\left(p,\mathfrak{m}^L\left(s_1;q\right)\right)}^{L}(X,Y)$. 
\item If $0<q\leq s\leq\infty$ and $0<p_{1}\leq p_{2}$, then $\Pi_{\left(p_{1},\mathfrak{m}^L\left(s;q\right)\right)}^{L}(X,Y)\subset\Pi_{\left(p_{2},m^L\left(s;q\right)\right)}^{L}(X,Y)$. Moreover $$\pi_{\left(p_2,\mathfrak{m}^L\left(s;q\right)\right)}^{L}(T)\leq\pi_{\left(p_1,\mathfrak{m}^L\left(s;q\right)\right)}^{L}(T)$$ for every $T\in\Pi_{\left(p_{1},\mathfrak{m}^L\left(s;q\right)\right)}^{L}(X,Y)$. 
\end{itemize}
\item Inspired by the dual operators of linear and nonlinear operators between arbitrary Banach spaces, see I.Sawashima \cite{S95}, the Lipschitz dual operator $S^{\#}$ from $Y^{\#}$ into $X^{\#}$ of a map $S\in\mathbb{L}_{x_{0}}(X,Y)$ is defined by $$\left\langle g, Sx\right\rangle_{(Y^{\#},Y)}=\left\langle S^{\#}g,x\right\rangle_{(X^{\#},X)}$$ for every $x\in X$ and $g\in Y^{\#}$. This is a bounded linear operator and $Lip(S)=\left\|S^{\#}\right\|_{\mathfrak{L}(Y^{\#},X^{\#})}$.
\item Recall Definition \ref{seven}. If $0<p\leq q$ and $s=q$, then every Lipschitz map is Lipschitz $(\mathfrak{m}^L\left(q;q\right),p)$\\$-$summing. If $0< q\leq\infty$ and $p=q=s$, then $\pi_{\left(\mathfrak{m}^L\left(q;q\right),q\right)}^{L}(I_{X})=1$, where $I_{X}$ stands for the identity map on $X$.
\item Recall Definition \ref{eight}. If $1\leq p<\infty$ and $p=s=q$, then the $\left(p,\mathfrak{m}^L\left(p;p\right)\right)-$summing maps are the Lipschitz $p-$summing maps considered in \cite{J09}. If $1\leq s< p$ and $s=q$, then the $\left(p,\mathfrak{m}^L\left(s;s\right)\right)-$summing maps are the Lipschitz $(p,s)-$summing considered in \cite{WG09}. It is also proved in \cite{WG09} that every Lipschitz map is Lipschitz $(p,s)-$summing.
\end{enumerate}
\end{Concluding Remarks}

\section{Main Results}
The following characterization of Lipschitz $\left(\mathfrak{m}^L\left(s;q\right),p\right)-$summing maps is presented in the following theorem, it is somewhat inspired by analogous results in the linear theory.

\begin{thm}\label{a}
A Lipschitz map $S$ from $X$ into $Y$ is Lipschitz $\left(\mathfrak{m}^L\left(s;q\right),p\right)-$summing if and only if there is a constant $C_{10}\geq 0$ such that 
\begin{equation}\label{bus}
\Bigg[\sum\limits_{j=1}^{m}\left|\sigma_j\right|^{q}\bigg[\sum\limits_{k=1}^{n}\left|\left\langle g_{k},Sx'_j\right\rangle_{(Y^{\#},Y)}-\left\langle  g_{k},Sx''_j\right\rangle_{(Y^{\#},Y)}\right|^{s}\bigg]^\frac{q}{s}\Bigg]^{\frac{1}{q}}
\leq C_{10}\cdot\norm{(\sigma,x',x'')\Big|\ell_p^{L,w}}\cdot\left\|(g_k)_{k=1}^{n}\Big|\ell_s(Y^{\#})\right\| 
\end{equation}
for every $\sigma_1,\cdot\cdot\cdot,\sigma_m\in\mathbb{R}$; $x'_1,\cdots, x'_m, x''_1,\cdots,x''_m\in X$; $g_1,\cdots,g_n\in Y^{\#}$ and $m,n\in\mathbb{N}$. Moreover 
$$\pi_{\left(\mathfrak{m}^L\left(s;q\right),p\right)}^{L}(S)=\inf C_{10}.$$
\end{thm}

\begin{proof} 
Assume that $S$ is a Lipschitz $\left(\mathfrak{m}^L\left(s;q\right),p\right)-$summing map. Consider $g_1,\cdot\cdot\cdot,g_n\in Y^{\#}$ and define the discrete probability  $\mu=\sum\limits_{k=1}^{n}t_k \delta_k$, where $t_k=Lip(g_k)^{s}\cdot\left\|(g_k)_{h=1}^{n}\Big|\ell_s(Y^{\#})\right\|^{-s}$ and $\delta_k$ denotes the Dirac measure at $b_k =\frac{g_k}{Lip(g_k)}\in B_{Y^{\#}}$; $k=1,\cdot\cdot\cdot,n$. Then $\mu\in W(B_{Y^{\#}})$. 

For $\sigma_1,\cdot\cdot\cdot,\sigma_m\in\mathbb{R}$, $x'_1,\cdots, x'_m, x''_1,\cdots,x''_m\in X$, we conclude from Proposition \ref{two} that
\begin{align}
\Bigg[\sum\limits_{j=1}^{m}\left|\sigma_j\right|^{q}&\bigg[\sum\limits_{k=1}^{n}\left|\left\langle g_{k},Sx'_j\right\rangle_{(Y^{\#},Y)}-\left\langle  g_{k},Sx''_j\right\rangle_{(Y^{\#},Y)}\right|^{s}\bigg]^\frac{q}{s}\Bigg]^{\frac{1}{q}} \nonumber \\ 
&=\Bigg[\sum\limits_{j=1}^{m}\left|\sigma_j\right|^{q}\bigg[\int\limits_{B_{Y^{\#}}}\left|\left\langle g,Sx'_j\right\rangle_{(Y^{\#},Y)}-\left\langle  g,Sx''_j\right\rangle_{(Y^{\#},Y)}\right|^{s}d\mu(g)\bigg]^\frac{q}{s}\Bigg]^{\frac{1}{q}}\cdot\left\|(g_k)_{k=1}^{n}\Big|\ell_s(Y^{\#})\right\| \nonumber \\
&\leq\mathfrak{m}_{(s;q)}^{L}(\sigma,Sx',Sx'')\cdot\left\|(g_k)_{k=1}^{n}\Big|\ell_s(Y^{\#})\right\|\nonumber \\
&\leq\pi_{\left(\mathfrak{m}^L\left(s;q\right),p\right)}^{L}(S)\cdot\left\|(\sigma,x',x'')\Big|\ell_p^{L,w}\right\|\cdot\left\|(g_k)_{k=1}^{n}\Big|\ell_s(Y^{\#})\right\|. \nonumber
\end{align} 
To show the converse, observe that (\ref{bus}) means 
\begin{align}\label{one}
\Bigg[\sum\limits_{j=1}^{m}\left|\sigma_j\right|^{q}&\bigg[\int\limits_{B_{Y^{\#}}}\left|\left\langle g,Sx'_j\right\rangle_{(Y^{\#},Y)}-\left\langle  g,Sx''_j\right\rangle_{(Y^{\#},Y)}\right|^{s}d\mu(g)\bigg]^\frac{q}{s}\Bigg]^{\frac{1}{q}}\leq C_{10}\cdot\left\|(\sigma,x',x'')\Big|\ell_p^{L,w}\right\| 
\end{align}
for every discrete probability measure $\mu$ on $B_{Y^{\#}}$ and $\sigma_1,\cdot\cdot\cdot,\sigma_m\in\mathbb{R}$; $x'_1,\cdots, x'_m, x''_1,\cdots,x''_m\in X$. 

Since the set of all finitely supported probability measures on $B_{Y^{\#}}$ is $\sigma(C(B_{Y^{\#}})^{\ast},C(B_{Y^{\#}}))-$ dense in the set of all probability measures on $B_{Y^{\#}}$, it follows that (\ref{one}) holds for all probability measures $\mu$ on $B_{Y^{\#}}$ and $\sigma_1,\cdot\cdot\cdot,\sigma_m\in\mathbb{R}$, $x'_1,\cdots, x'_m, x''_1,\cdots,x''_m\in X$. 

Taking the supremum over $\mu\in W(B_{Y^{\#}})$ on the left side of (\ref{one}) and using Proposition \ref{two}, we obtain $$\mathfrak{m}_{(s;q)}^{L}(\sigma,Sx',Sx'')\leq C_{10}\cdot\left\|(\sigma,x',x'')\Big|\ell_p^{L,w}\right\|.$$
\textcolor{red}{$\blacksquare$}
\end{proof}

In the previous section we obtained that for any Lipschitz mapping $S$ between pointed metric spaces $X$ and $Y$ one can naturally define a Lipschitz dual operator acting between the Lipschitz dual spaces $Y^\#$ and $X^\#$. The next theorem connects Lipschitz $\left(\mathfrak{m}^L\left(s;q\right),p\right)-$summing maps with the type constants of the Lipschitz dual operators, which is a well-known linear concept, see \cite{P80}.

\begin{thm}\label{type10}
Let $0<p\leq q<s<2$. If the Lipschitz dual operator $S^{\#}\in\mathfrak{L}(Y^{\#},X^{\#})$ of the map $S\in\mathbb{L}_{x_{0}}(X,Y)$ is $(s,p)-$type, then $S$ is a Lipschitz $\left(\mathfrak{m}^L\left(s;q\right),p\right)-$summing map. Moreover $$\pi_{\left(\mathfrak{m}^L\left(s;q\right),p\right)}^{L}(S)\leq\mathbf{T}_{(s,p)}(S^{\#}).$$
\end{thm}

\begin{proof}
Let $\sigma_1,\cdot\cdot\cdot,\sigma_m\in\mathbb{R}$; $x'_1,\cdots, x'_m, x''_1,\cdots,x''_m\in X$ and $g_1,\cdots,g_n\in Y^{\#}$. Then from (\ref{type1}) and $p\leq q$ we have

\begin{align}
&\Bigg[\sum\limits_{j=1}^{m}\left|\sigma_j\right|^{q}\bigg[\sum\limits_{k=1}^{n}\left|\left\langle g_{k},Sx'_j\right\rangle_{(Y^{\#},Y)}-\left\langle  g_{k},Sx''_j\right\rangle_{(Y^{\#},Y)}\right|^{s}\bigg]^\frac{q}{s}\Bigg]^{\frac{1}{q}}\nonumber \\
&\leq\Bigg[\sum\limits_{j=1}^{m}\left|\sigma_j\right|^{p}\bigg[\sum\limits_{k=1}^{n}\left|\left\langle g_{k},Sx'_j\right\rangle_{(Y^{\#},Y)}-\left\langle  g_{k},Sx''_j\right\rangle_{(Y^{\#},Y)}\right|^{s}\bigg]^\frac{p}{s}\Bigg]^{\frac{1}{p}}\nonumber \\
&=c_{sp}^{-1}\cdot\Bigg[\sum\limits_{j=1}^{m}\left|\sigma_j\right|^{p}\int\limits_{\mathbb{R}^{n}}\left|\sum\limits_{k=1}^{n}t_{k}\cdot\left\langle g_{k},Sx'_j\right\rangle_{(Y^{\#},Y)}-\sum\limits_{k=1}^{n}t_{k}\cdot\left\langle  g_{k},Sx''_j\right\rangle_{(Y^{\#},Y)}\right|^{p} d\mu_{s}^{n}(t)\Bigg]^{\frac{1}{p}}\nonumber \\
&=c_{sp}^{-1}\cdot\Bigg[\int\limits_{\mathbb{R}^{n}}\sum\limits_{j=1}^{m}\left|\sigma_j\right|^{p}\left|\sum\limits_{k=1}^{n}t_{k}\cdot\left\langle S^{\#}g_{k},x'_j\right\rangle_{(X^{\#},X)}-\sum\limits_{k=1}^{n}t_{k}\cdot\left\langle  S^{\#}g_{k},x''_j\right\rangle_{(X^{\#},X)}\right|^{p} d\mu_{s}^{n}(t)\Bigg]^{\frac{1}{p}}\nonumber \\
&=c_{sp}^{-1}\cdot\Bigg[\int\limits_{\mathbb{R}^{n}}\sum\limits_{j=1}^{m}\left|\sigma_j\right|^{p}\left|\left\langle\sum\limits_{k=1}^{n}t_{k}S^{\#}g_{k},x'_j\right\rangle_{(X^{\#},X)}-\left\langle\sum\limits_{k=1}^{n}t_{k}S^{\#}g_{k},x''_j\right\rangle_{(X^{\#},X)}\right|^{p} d\mu_{s}^{n}(t)\Bigg]^{\frac{1}{p}}\nonumber \\
&\leq c_{sp}^{-1}\cdot\Bigg[\int\limits_{\mathbb{R}^{n}} Lip\left(\sum\limits_{k=1}^{n}t_{k}S^{\#}g_{k}\right)^{p} d\mu_{s}^{n}(t)\Bigg]^{\frac{1}{p}}\cdot\left\|(\sigma,x',x'')\Big|\ell_p^{L,w}\right\|.\nonumber 
\end{align}
Hence from (\ref{type2}) and (\ref{type3}) we obtain
\begin{align}
\Bigg[\sum\limits_{j=1}^{m}\left|\sigma_j\right|^{q}&\bigg[\sum\limits_{k=1}^{n}\left|\left\langle g_{k},Sx'_j\right\rangle_{(Y^{\#},Y)}-\left\langle  g_{k},Sx''_j\right\rangle_{(Y^{\#},Y)}\right|^{s}\bigg]^\frac{q}{s}\Bigg]^{\frac{1}{q}}\nonumber \\
&\leq t_{(s,p)}(S^{\#}g_{k})\cdot\left\|(\sigma,x',x'')\Big|\ell_p^{L,w}\right\|\nonumber \\
&\leq\mathbf{T}_{(s,p)}(S^{\#})\cdot\left\|(\sigma,x',x'')\Big|\ell_p^{L,w}\right\|\cdot\left\|(g_k)_{k=1}^{n}\Big|\ell_s(Y^{\#})\right\|.\nonumber
\end{align}
\end{proof}

In the case that $p=q$, Theorem \ref{type10} reads as follows.

\begin{Corollary}\label{type5}
Let $0<p<s<2$. If the Lipschitz dual operator $S^{\#}\in\mathfrak{L}(Y^{\#},X^{\#})$ of the map $S\in\mathbb{L}_{x_{0}}(X,Y)$ is of $(s,p)-$type, then $S$ is a Lipschitz $\left(\mathfrak{m}^L\left(s;p\right),p\right)-$summing map. Moreover $$\pi_{\left(\mathfrak{m}^L\left(s;p\right),p\right)}^{L}(S)\leq\mathbf{T}_{(s,p)}(S^{\#}).$$
\end{Corollary}

The composition result that will be used later is the following:

\begin{thm}\label{manaf1}
Let $0<p\leq s\leq\infty$, $p\geq r$. If $S\in\Pi_{\left(\mathfrak{m}^L\left(s;p\right),r\right)}^{L}(X,Y)$ and $T\in\Pi_{\left(s,\mathfrak{m}^L\left(s;s\right)\right)}^{L}(Y,Z)$, then $T\circ S\in\Pi_{\left(p,\mathfrak{m}^L\left(r;r\right)\right)}^{L}(X,Z)$. Moreover
$$\pi_{\left(p,\mathfrak{m}^L\left(r;r\right)\right)}^{L}(T\circ S)\leq\pi_{\left(s,\mathfrak{m}^L\left(s;s\right)\right)}^{L}(T)\cdot \pi_{\left(\mathfrak{m}^L\left(s;p\right),r\right)}^{L}(S).$$
\end{thm}

\begin{proof}
We recall that for $0<p\leq s\leq\infty$ the index $s'\left(p\right)$ is determined by the equation $$\frac{1}{s'\left(p\right)}+\frac{1}{s}=\frac{1}{p}.$$ Let $\sigma_1,\cdot\cdot\cdot,\sigma_m\in\mathbb{R}$, $x'_1,\cdots, x'_m, x''_1,\cdots,x''_m\in X$ and $m\in\N$. 
The H\"older inequality, the definition of Lipschitz $\left(s,\mathfrak{m}^L\left(s;s\right)\right)-$summing maps, the definition of $\left(\mathfrak{m}^L\left(s;p\right),r\right)-$summing maps and (\ref{three}) naturally come together to give us
\begin{align} \label{manaf}
\norm{\left(\sigma,(T\circ S)x',(T\circ S)x''\right)\Big|\ell_p}&\leq\left\|\tau\Big|\ell_{s'\left(p\right)}\right\|\cdot\left\|\left(\frac{\sigma}{\tau},T(Sx'),T(Sx'')\right)\Big|\ell_s\right\| \nonumber \\
&\leq\pi_{\left(s,\mathfrak{m}^L\left(s;s\right)\right)}^{L}(T)\cdot\left\|\tau\Big|\ell_{s'\left(p\right)}\right\|\cdot\left\|\Big(\frac{\sigma}{\tau},Sx',Sx''\Big)\Big|\ell_s^{L,w}\right\|.  
\end{align}  
Taking the infimum over all $\tau\in\ell_{s'\left(p\right)}$ on the right side of (\ref{manaf}), we get
\begin{align}
\left\|\left(\sigma,(T\circ S)x',(T\circ S)x''\right)\Big|\ell_p\right\|&\leq\pi_{\left(s,\mathfrak{m}^L\left(s;s\right)\right)}^{L}(T)\cdot\mathfrak{m}_{(s;p)}^{L}(\sigma,Sx',Sx'') \nonumber \\
&\leq\pi_{\left(s,\mathfrak{m}^L\left(s;s\right)\right)}^{L}(T)\cdot\pi_{\left(\mathfrak{m}^L\left(s;p\right),r\right)}^{L}(S)\cdot\left\|(\sigma,x',x'')\Big|\ell_r^{L,w}\right\|. \nonumber
\end{align}
Hence $T\circ S\in\Pi_{\left(p,\mathfrak{m}^L\left(r;r\right)\right)}^{L}(X,Z)$. Now it follows from Definition \ref{eight} that $$\pi_{\left (p,\mathfrak{m}^L\left(r;r\right)\right)}^{L}(T\circ S)\leq\pi_{\left(s,\mathfrak{m}^L\left(s;s\right)\right)}^{L}(T)\cdot\pi_{\left(\mathfrak{m}^L\left(s;p\right),r\right)}^{L}(S).$$
\end{proof}
The special case $p=r$ gives.

\begin{Corollary}\label{manaf100}
Let $0<p\leq s\leq\infty$. If $S\in\Pi_{\left(\mathfrak{m}^L\left(s;p\right),p\right)}^{L}(X,Y)$ and $T\in\Pi_{\left(s,\mathfrak{m}^L\left(s;s\right)\right)}^{L}(Y,Z)$, then $T\circ S\in\Pi_{\left(p,\mathfrak{m}^L\left(p;p\right)\right)}^{L}(X,Z)$. Moreover
$$\pi_{\left(p,\mathfrak{m}^L\left(p;p\right)\right)}^{L}(T\circ S)\leq\pi_{\left(s,\mathfrak{m}^L\left(s;s\right)\right)}^{L}(T)\cdot\pi_{\left(\mathfrak{m}^L\left(s;p\right),p\right)}^{L}(S).$$
\end{Corollary}
 
\begin{Lemma}
Let $0<q\leq s\leq\infty$, $p\leq q$ and let $T$ be a Lipschitz map from $X$ into $Y$. If there is a constant $C_{11}\geq 0$ such that for any probability measure $\nu$ on $B_{{Y}^{\#}}$ there exists a probability measure $\mu$ on $B_{{X}^{\#}}$ such that for every $a$, $b$ in $X$, $$\Bigg[\int\limits_{B_{Y^{\#}}}\left|g(Ta)-g(Tb)\right|^{s}d\upsilon(g)\Bigg]^\frac{1}{s}\leq C_{11}\cdot\Bigg[\int\limits_{B_{X^{\#}}}\left|f(a)-f(b)\right|^{p}d\mu(f)\Bigg]^\frac{1}{p},$$ then $T$ is a Lipschitz $\left(\mathfrak{m}^L\left(s;q\right),p\right)-$summing map. Moreover 
$$\pi_{\left(\mathfrak{m}^L\left(s;q\right),p\right)}^{L}(T)=\inf C_{11}.$$
\end{Lemma}

\begin{proof}
Let $\eta$ be an arbitrary sequence in $\mathbb{R}$. By the assumptions, we have
\begin{align}\label{five}
\Bigg[\sum\limits_{j=1}^{m}\bigg[\int\limits_{B_{Y^{\#}}}\left|\eta_j\right|^{s}\left|g(Ta_j)-g(Tb_j)\right|^{s}d\nu(g)\bigg]^\frac{q}{s}\Bigg]^\frac{1}{q}&\leq\Bigg[\sum\limits_{j=1}^{m}\bigg[\int\limits_{B_{Y^{\#}}}\left|\eta_j\right|^{s}\left|g(Ta_j)-g(Tb_j)\right|^{s}d\nu(g)\bigg]^\frac{p}{s}\Bigg]^\frac{1}{p}\nonumber \\
&\leq C_{11}\cdot\Bigg[\sum\limits_{j=1}^{m}\int\limits_{B_{X^{\#}}}\left|\eta_j\right|^{p}\left|f(a_j)-f(b_j)\right|^{p}d\mu(f)\Bigg]^\frac{1}{p}\nonumber \\
&\leq C_{11}\cdot\left\|(\eta,a,b)\Big|\ell_p^{L,w}\right\|.
\end{align}
Taking the supremum over $\nu\in B_{Y^{\#}}$ on the left side of (\ref{five}) and from Proposition \ref{two}, we get
$$\mathfrak{m}_{(s;q)}^{L}(\eta,Ta,Tb)\leq C_{11}\cdot\left\|(\eta,a,b)\Big|\ell_p^{L,w}\right\|$$
\end{proof} 

\begin{Proposition}\label{hous}
Let $0<q\leq r\leq t\leq\infty$. If $S$ from $Y$ into $Z$ is a Lipschitz $\left(\mathfrak{m}^L\left(t;s\right),r\right)-$summing map and $T$ from $X$ into $Y$ is a Lipschitz $\left(\mathfrak{m}^L\left(r;p\right),q\right)-$summing map, then $S\circ T$ from $X$ into $Z$ is a Lipschitz $\left(\mathfrak{m}^L\left(t;p\right),q\right)-$summing map. Moreover $$\pi_{\left(\mathfrak{m}^L\left(t;p\right),q\right)}^{L}(S\circ T)\leq\pi_{\left(\mathfrak{m}^L\left(t;s\right),r\right)}^{L}(S)\cdot\pi_{\left(\mathfrak{m}^L\left(r;p\right),q\right)}^{L}(T).$$
\end{Proposition}

\begin{proof}
From Definition \ref{car}, we have
\begin{align}
\mathfrak{m}_{(t;p)}^{L}\big(\sigma,(S\circ T)x',(S\circ T) x''\big)&=\inf\limits_{\tau}\left\|\tau\Big|\ell_{t'\left(p\right)}\right\|\left\|\Big(\frac{\sigma}{\tau},(S\circ T)x',(S\circ T)x''\Big)\Big|\ell_t^{L,w}\right\|\nonumber \\
&=\inf\limits_{\tau_{1}\cdot\tau_{2}}\left\|\tau_{1}\cdot\tau_{2}\Big|\ell_{t'\left(p\right)}\right\|\left\|\Big(\frac{\sigma}{\tau_{1}\cdot\tau_{2}},(S\circ T)x',(S\circ T)x''\Big)\Big|\ell_t^{L,w}\right\|.\nonumber
\end{align}
Let $\sigma'=\frac{\sigma}{\tau_{1}}$. Since $\frac{1}{t'\left(s\right)}+\frac{1}{r'\left(p\right)}=\frac{1}{t'\left(p\right)}$ with the H\"older inequality give us
\begin{align}
\mathfrak{m}_{(t;p)}^{L}\big(\sigma,(S\circ T)x',(S\circ T) x''\big)&\leq\inf\limits_{\tau_{1}\cdot\tau_{2}}\left\|\tau_{1}\Big|\ell_{r'\left(p\right)}\right\|\cdot\left\|\tau_{2}\Big|\ell_{t'\left(s\right)}\right\|\left\|\Big(\frac{\sigma'}{\tau_{2}},S(Tx'),S(Tx'')\Big)\Big|\ell_t^{L,w}\right\|\nonumber \\
&=\inf\limits_{\tau_{1}}\left\|\tau_{1}\Big|\ell_{r'\left(p\right)}\right\|\cdot\inf\limits_{\tau_{2}}\left\|\tau_{2}\Big|\ell_{t'\left(s\right)}\right\|\left\|\Big(\frac{\sigma'}{\tau_{2}},S(Tx'),S(Tx'')\Big)\Big|\ell_t^{L,w}\right\|\nonumber \\
&=\inf\limits_{\tau_{1}}\left\|\tau_{1}\Big|\ell_{r'\left(p\right)}\right\|\cdot\mathfrak{m}_{(t;s)}^{L}\big(\sigma',S(Tx'),S(Tx'')\big)\nonumber \\
&\leq\pi_{\left(\mathfrak{m}^L\left(t;s\right),r\right)}^{L}(S)\cdot\inf\limits_{\tau_{1}}\left\|\tau_{1}\Big|\ell_{r'\left(p\right)}\right\|\cdot\left\|(\sigma',Tx',Tx'')\Big|\ell_r^{L,w}\right\|\nonumber \\
&=\pi_{\left(\mathfrak{m}^L\left(t;s\right),r\right)}^{L}(S)\cdot\mathfrak{m}_{(r;p)}^{L}(\sigma,Tx',Tx'')\nonumber \\
&\leq\pi_{\left(\mathfrak{m}^L\left(t;s\right),r\right)}^{L}(S)\cdot\pi_{\left(\mathfrak{m}^L\left(r;p\right),q\right)}^{L}(T)\cdot\left\|(\sigma,x',x'')\Big|\ell_q^{L,w}\right\|.\nonumber
\end{align}

Finally, it follows from Definition \ref{seven} that $$\pi_{\left(\mathfrak{m}^L\left(t;p\right),q\right)}^{L}(S\circ T)\leq\pi_{\left(\mathfrak{m}^L\left(t;s\right),r\right)}^{L}(S)\cdot\pi_{\left(\mathfrak{m}^L\left(r;p\right),q\right)}^{L}(T).$$
\end{proof}

In the case that $s=r$ and $p=q$, this result gives

\begin{Corollary}
Let $0<p\leq s\leq t\leq\infty$. If $S$ from $Y$ into $Z$ is a Lipschitz $\left(\mathfrak{m}^L\left(t;s\right),s\right)-$summing map and $T$ from $X$ into $Y$ is a Lipschitz $\left(\mathfrak{m}^L\left(s;p\right),p\right)-$summing map, then $S\circ T$ from $X$ into $Z$ is a Lipschitz $\left(\mathfrak{m}^L\left(t;p\right),p\right)-$summing map. Moreover $$\pi_{\left(\mathfrak{m}^L\left(t;p\right),p\right)}^{L}(S\circ T)\leq\pi_{\left(\mathfrak{m}^L\left(t;s\right),s\right)}^{L}(S)\cdot\pi_{\left(\mathfrak{m}^L\left(s;p\right),p\right)}^{L}(T).$$
\end{Corollary}

An interesting inclusion result that will be used later is the following.

\begin{thm}\label{d}
If $0<p\leq s$, then $\Pi_{\left(p,\mathfrak{m}^L\left(s;p\right)\right)}^{L}(X,Y)\subset\Pi_{\left(s,\mathfrak{m}^L\left(s;s\right)\right)}^{L}(X,Y).$ Moreover $$\pi_{\left(s,\mathfrak{m}^L\left(s;s\right)\right)}^{L}(T)\leq\pi_{\left(p,\mathfrak{m}^L\left(s;p\right)\right)}^{L}(T)$$ for every $T\in\Pi_{\left(p,\mathfrak{m}^L\left(s;p\right)\right)}^{L}(X,Y)$.
\end{thm}

\begin{proof}
Let $T$ be an arbitrary operator in $\Pi_{\left(p,\mathfrak{m}^L\left(s;p\right)\right)}^{L}(X,Y)$; $\sigma_j, \alpha_j\in\mathbb{R}$; $x'_j$, $x''_j\in X$; $j=1,\cdot\cdot\cdot,m$ and $m\in\mathbb{N}$. We have
\begin{align}\label{salah}
\sum\limits_{j=1}^{m}\left|\alpha_j\right|\left|\sigma_j\right|^{p} d_Y(Tx'_j,Tx''_j)^{p}&=\sum\limits_{j=1}^{m}\left|\left|\alpha_j\right|^\frac{1}{p}\cdot\sigma_j\right|^{p} d_Y(Tx'_j,Tx''_j)^{p} \nonumber \\ 
&\leq\Big[\pi_{\left(p,\mathfrak{m}^L\left(s;p\right)\right)}^{L}(T)\Big]^p\cdot\Big[\mathfrak{m}_{(s;p)}^{L}\big(\left|\alpha\right|^\frac{1}{p}\cdot\sigma,x',x''\big)\Big]^p \nonumber \\ 
&\leq\Big[\pi_{\left(p,\mathfrak{m}^L\left(s;p\right)\right)}^{L}(T)\Big]^{p}\cdot\norm{\alpha\Big|\ell_\frac{s'\left(p\right)}{p}}\cdot\norm{(\sigma,x',x'')\Big|\ell_s^{L,w}}^{p}. 
\end{align}
Taking the supremum over all such $\alpha$ with $\cdot\norm{\alpha\Big|\ell_\frac{s'\left(p\right)}{p}}\leq 1$ on the both sides of (\ref{salah}), we have
$$\Bigg[\sum\limits_{j=1}^{m}\bigg[\left|\sigma_j\right|^{p} d_Y(Tx'_j,Tx''_j)^{p}\bigg]^\frac{s}{p}\Bigg]^\frac{p}{s}\leq\Big[\pi_{\left(p,\mathfrak{m}^L\left(s;p\right)\right)}^{L}(T)\Big]^p\cdot\left\|(\sigma,x',x'')\Big|\ell_s^{L,w}\right\|^{p}.$$ 
But this implies that
$$\left\|(\sigma,Tx',Tx'')\Big|\ell_s\right\|\leq\pi_{\left(p,\mathfrak{m}^L\left(s;p\right)\right)}^{L}(T)\cdot\left\|(\sigma,x',x'')\Big|\ell_s^{L,w}\right\|.$$ Hence $T$ is a Lipschitz $\left(s,\mathfrak{m}^L\left(s;s\right)\right)-$summing map. Finally, it follows from Definition \ref{eight} that $$\pi_{\left(s,\mathfrak{m}^L\left(s;s\right)\right)}^{L}(T)\leq\pi_{\left(p,\mathfrak{m}^L\left(s;p\right)\right)}^{L}(T).$$
\end{proof} 

More generally, we have the following composition result between Lipschitz $\left(p,\mathfrak{m}^L\left(s;q\right)\right)$ and Lipschitz $\left(\mathfrak{m}^L\left(s;q\right),r\right)-$summing maps.

\begin{Proposition}
Let $0<q\leq s\leq\infty$, $q\geq r$ and $p\geq q$. If $S\in\Pi_{\left(\mathfrak{m}^L\left(s;q\right),r\right)}^{L}(X,Y)$ and $T\in\Pi_{\left(p,\mathfrak{m}^L\left(s;q\right)\right)}^{L}(Y,Z)$, then $T\circ S\in\Pi_{\left(p,\mathfrak{m}^L\left(r;r\right)\right)}^{L}(X,Z)$. Moreover
$$\pi_{\left(p,\mathfrak{m}^L\left(r;r\right)\right)}^{L}(T\circ S)\leq\pi_{\left(p,\mathfrak{m}^L\left(s;q\right)\right)}^{L}(T)\cdot\pi_{\left(\mathfrak{m}^L\left(s;q\right),r\right)}^{L}(S).$$
\end{Proposition}

\begin{proof}
Together, Definition \ref{seven} and Definition \ref{eight} immediately give us the result.\\
\end{proof}
In the case that $p=q$, this result gives

\begin{Corollary}\label{b}
Let $0<p\leq s\leq\infty$, $p\geq r$. If $S\in\Pi_{\left(\mathfrak{m}^L\left(s;p\right),r\right)}^{L}(X,Y)$ and $T\in\Pi_{\left(p,\mathfrak{m}^L\left(s;p\right)\right)}^{L}(Y,Z)$, then $T\circ S\in\Pi_{\left(p,\mathfrak{m}^L\left(r;r\right)\right)}^{L}(X,Z)$. Moreover
$$\pi_{\left(p,\mathfrak{m}^L\left(r;r\right)\right)}^{L}(T\circ S)\leq\pi_{\left(p,\mathfrak{m}^L\left(s;p\right)\right)}^{L}(T)\cdot\pi_{\left(\mathfrak{m}^L\left(s;p\right),r\right)}^{L}(S).$$
\end{Corollary}

\begin{thm}\label{c}
Let $1\leq s\leq\infty$, $0< p\leq s\leq\infty$ and $p\geq r$. If $S$ is a Lipschitz map from $X$ into $Y$ such that $T\circ S\in\Pi_{\left(p,\mathfrak{m}^L\left(r;r\right)\right)}^{L}(X,F)$ for every $T\in\Pi_{\left (p,\mathfrak{m}^L\left(s;p\right)\right)}^{L}(Y,F)$ and each Banach space $F$, then $S\in\Pi_{\left(\mathfrak{m}^L\left(s;p\right),r\right)}^{L}(X,Y)$. Moreover $$\pi_{\left(\mathfrak{m}^L\left(s;p\right),r\right)}^{L}(S)=\sup\limits_{F\ Banach\ space}\left\{\pi_{\left(p,\mathfrak{m}^L\left(r;r\right)\right)}^{L}(T\circ S):T\in\Pi_{\left (p,\mathfrak{m}^L\left(s;p\right)\right)}^{L}(Y,F);\pi_{\left (p,\mathfrak{m}^L\left(s;p\right)\right)}^{L}(T)\leq 1\right\}$$ 
\end{thm}

\begin{proof}
By analogous reasoning as in the theory of operator ideals, see \cite[Sec. 7.2]{P80} we have \\$M=\sup\limits_{F\ Banach\ space}\left\{\pi_{\left(p,\mathfrak{m}^L\left(r;r\right)\right)}^{L}(T\circ S):T\in\Pi_{\left (p,\mathfrak{m}^L\left(s;p\right)\right)}^{L}(Y,F);\pi_{\left (p,\mathfrak{m}^L\left(s;p\right)\right)}^{L}(T)\leq 1\right\}<\infty$.\\ For $(g_k)_{k=1}^{n}\subset Y^{\#}$ and $n\in\mathbb{N}$. We define $T\in\Pi_{\left(p,\mathfrak{m}^L\left(s;p\right)\right)}^{L}(Y,\ell_s)$ by the rule $$T(y)=\Big(\left\langle g_{1},y\right\rangle_{(Y^{\#},Y)},\cdots,\left\langle g_{n},y\right\rangle_{(Y^{\#},Y)},0,0,0,\cdots\Big).$$ We recall that for $0<p\leq s\leq\infty$, the index $s'\left(p\right)$ is determined by the equation $$\frac{1}{s'\left(p\right)}+\frac{1}{s}=\frac{1}{p}.$$ Let $\sigma_1,\cdot\cdot\cdot,\sigma_m\in\mathbb{R}$; $x'_1,\cdots, x'_m, x''_1,\cdots,x''_m\in X$; $y'_1,\cdot\cdot\cdot,y'_m, y''_1,\cdot\cdot\cdot,y''_m\in Y$ and $m\in\N$. By using the H\"older inequality we have

\begin{align}\label{manaf4}
\Bigg[\sum\limits_{j=1}^{m}&\left|\sigma_j\right|^{p}\left\|Ty'_j-Ty''_j\Big|\ell_s\right\|^{p}\Bigg]^\frac{1}{p}\leq\left\|\tau\Big|\ell_{s'\left(p\right)}\right\|\cdot\Bigg[\sum\limits_{j=1}^{m}\left|\frac{\sigma_j}{\tau_j}\right|^{s}\left\|Ty'_j-Ty''_j\Big|\ell_s\right\|^{s}\Bigg]^\frac{1}{s} \nonumber \\ 
&=\left\|\tau\Big|\ell_{s'\left(p\right)}\right\|\cdot\Bigg[\sum\limits_{j=1}^{m}\left|\frac{\sigma_j}{\tau_j}\right|^{s}\left\|\Big(\left\langle g_k,y'_j\right\rangle_{(Y^{\#},Y)}-\left\langle g_k,y''_j\right\rangle_{(Y^{\#},Y)}\Big)_{k=1}^{n}\Big|\ell_s\right\|^{s}\Bigg]^\frac{1}{s} \nonumber \\ 
&=\left\|\tau\Big|\ell_{s'\left(p\right)}\right\|\cdot\Bigg[\sum\limits_{j=1}^{m}\sum\limits_{k=1}^{n}\left|\frac{\sigma_j}{\tau_j}\right|^{s}\left|\left\langle g_k,y'_j\right\rangle_{(Y^{\#},Y)}-\left\langle g_k,y''_j\right\rangle_{(Y^{\#},Y)}\right|^{s}\Bigg]^\frac{1}{s} \nonumber \\
&=\left\|\tau\Big|\ell_{s'\left(p\right)}\right\|\cdot\Bigg[\sum\limits_{j=1}^{m}\sum\limits_{k=1}^{n}\left|\frac{\sigma_j}{\tau_j}\right|^{s}\cdot Lip(g_k)^s\left|\left\langle\frac{g_k}{Lip(g_k)},y'_j\right\rangle_{(Y^{\#},Y)}-\left\langle\frac{g_k}{Lip(g_k)},y''_j\right\rangle_{(Y^{\#},Y)}\right|^{s}\Bigg]^\frac{1}{s} \nonumber \\ 
&=\left\|\tau\Big|\ell_{s'\left(p\right)}\right\|\cdot\Bigg[\sum\limits_{k=1}^{n}\sum\limits_{j=1}^{m}\left|\frac{\sigma_j}{\tau_j}\right|^{s}\cdot Lip(g_k)^s\left|\left\langle\frac{g_k}{Lip(g_k)},y'_j\right\rangle_{(Y^{\#},Y)}-\left\langle\frac{g_k}{Lip(g_k)},y''_j\right\rangle_{(Y^{\#},Y)}\right|^{s}\Bigg]^\frac{1}{s} \nonumber \\
&\leq\left\|\tau\Big|\ell_{s'\left(p\right)}\right\|\cdot\left\|(g_k)_{k=1}^{n}\Big|\ell_s(Y^{\#})\right\|\cdot\left\|\big(\frac{\sigma}{\tau},y',y''\big)\Big|\ell_s^{L,w}\right\|. 
\end{align}
Taking the infimum over all $\tau\in\ell_{s'\left(p\right)}$ on the right side of (\ref{manaf4}), we get 
$$\Bigg[\sum\limits_{j=1}^{m}\left|\sigma_j\right|^{p}\left\|Ty'_j-Ty''_j\big|\ell_s\right\|^{p}\Bigg]^\frac{1}{p}\leq\left\|(g_k)_{k=1}^{n}\Big|\ell_s(Y^{\#})\right\|\cdot \mathfrak{m}_{(s;p)}^{L}(\sigma,y',y'').$$
Then 
$$
 \pi_{\left(p,\mathfrak{m}^L\left(s;p\right)\right)}^{L}(T)\leq\left\|(g_k)_{k=1}^{n}\Big|\ell_s(Y^{\#})\right\|.
$$
\begin{align}\label{moon}
\Bigg[\sum\limits_{j=1}^{m}\left|\sigma_j\right|^{p} & 
\bigg[\sum\limits_{k=1}^{n}\left|\left\langle g_k,Sx'_j\right\rangle_{(Y^{\#},Y)}-\left\langle g_k,Sx''_j\right\rangle_{(Y^{\#},Y)}\right|^{s}\bigg]^\frac{p}{s}\Bigg]^\frac{1}{p}\nonumber \\
&=\Bigg[\sum\limits_{j=1}^{m}\left|\sigma_j\right|^{p}\left\|\Big(\left\langle g_k,Sx'_j\right\rangle_{(Y^{\#},Y)}-\left\langle g_k,Sx''_j\right\rangle_{(Y^{\#},Y)}\Big)_{k=1}^{n}\Big|\ell_s\right\|^{p}\Bigg]^\frac{1}{p}\nonumber \\
&=\Bigg[\sum\limits_{j=1}^{m}\left|\sigma_j\right|^{p}\left\|T(Sx'_j)-T(Sx''_j)\Big|\ell_s\right\|^{p}\Bigg]^\frac{1}{p} \nonumber \\
&=\Bigg[\sum\limits_{j=1}^{m}\left|\sigma_j\right|^{p}\left\|(T\circ S)(x'_j)-(T\circ S)(x''_j)\Big|\ell_s\right\|^{p}\Bigg]^\frac{1}{p} \\
&\leq\pi_{\left(p,\mathfrak{m}^L\left(r;r\right)\right)}^{L}(T\circ S)\cdot\left\|(\sigma,x',x'')\Big|\ell_r^{L,w}\right\| \nonumber\\
&\leq M\cdot\left\|(g_k)_{k=1}^{n}\Big|\ell_s(Y^{\#})\right\|\cdot\left\|(\sigma,x',x'')\Big|\ell_r^{L,w}\right\| \nonumber 
\end{align}
By Theorem \ref{a}, we obtain $S\in\Pi_{\left(\mathfrak{m}^L\left(s;p\right),r\right)}^{L}(X,Y)$ with $\pi_{\left(\mathfrak{m}^L\left(s;p\right),r\right)}^{L}(S)\leq M$. From Corollary \ref{b} we have $\pi_{\left(\mathfrak{m}^L\left(s;p\right),r\right)}^{L}(S)=M$. \\ 
\end{proof}
 
Now we consider some special cases. By Theorem \ref{d} and Theorem \ref{c}, we can state the following result.

\begin{Corollary}\label{e}
Let $1\leq s\leq\infty$, $0< p\leq s\leq\infty$, $p\geq r$. If $S$ is a Lipschitz map from $X$ into $Y$ such that $T\circ S\in\Pi_{\left(p,\mathfrak{m}^L\left(r;r\right)\right)}^{L}(X,F)$ for every $T\in\Pi_{\left (s,\mathfrak{m}^L\left(s;s\right)\right)}^{L}(Y,F)$ and each Banach space $F$, then $S\in\Pi_{\left(\mathfrak{m}^L\left(s;p\right),r\right)}^{L}(X,Y)$. Moreover $$\pi_{\left(\mathfrak{m}^L\left(s;p\right),r\right)}^{L}(S)=\sup\limits_{F\ Banach\ space}\left\{\pi_{\left(p,\mathfrak{m}^L\left(r;r\right)\right)}^{L}(T\circ S):T\in\Pi_{\left (s,\mathfrak{m}^L\left(s;s\right)\right)}^{L}(Y,F);\pi_{\left (s,\mathfrak{m}^L\left(s;s\right)\right)}^{L}(T)\leq 1\right\}$$ 
\end{Corollary}

\begin{proof}
By analogous reasoning as in the theory of operator ideals, see \cite[Sec. 7.2]{P80} we have\\
$D=\sup\limits_{F\ Banach\ space}\left\{\pi_{\left(p,\mathfrak{m}^L\left(r;r\right)\right)}^{L}(T\circ S):T\in\Pi_{\left (s,\mathfrak{m}^L\left(s;s\right)\right)}^{L}(Y,F);\pi_{\left (s,\mathfrak{m}^L\left(s;s\right)\right)}^{L}(T)\leq 1\right\}<\infty$.\\ As in the beginning of the proof of Theorem \ref{c}, for $(g_k)_{k=1}^{n}\subset Y^{\#}$ we define $T\in\Pi_{\left(p,\mathfrak{m}^L\left(s;p\right)\right)}^{L}(Y,\ell_s)$ by the rule $$T(y)=\Big(\left\langle g_{1},y\right\rangle_{(Y^{\#},Y)},\cdots,\left\langle g_{n},y\right\rangle_{(Y^{\#},Y)},0,0,0,\cdots\Big).$$ We also obtained 
$$\pi_{\left(p,\mathfrak{m}^L\left(s;p\right)\right)}^{L}(T)\leq\left\|(g_k)_{k=1}^{n}\Big|\ell_s(Y^{\#})\right\|.$$ It follows from Theorem \ref{d} that $$\pi_{\left(s,\mathfrak{m}^L\left(s;s\right)\right)}^{L}(T)\leq\left\|(g_k)_{k=1}^{n}\Big|\ell_s(Y^{\#})\right\|.$$
Then from (\ref{moon}) we have
\begin{align}
\Bigg[\sum\limits_{j=1}^{m}\left|\sigma_j\right|^{p}\bigg[\sum\limits_{k=1}^{n}&\left|\left\langle g_k,Sx'_j\right\rangle_{(Y^{\#},Y)}-\left\langle  g_k,Sx''_j\right\rangle_{(Y^{\#},Y)}\right|^{s}\bigg]^\frac{p}{s}\Bigg]^\frac{1}{p}\nonumber \\
&\leq D\cdot\left\|(g_k)_{k=1}^{n}\Big|\ell_s(Y^{\#})\right\|\cdot\left\|(\sigma,x',x'')\Big|\ell_r^{L,w}\right\|. \nonumber
\end{align}
By Theorem \ref{a}, we obtain $S\in\Pi_{\left(\mathfrak{m}^L\left(s;p\right),r\right)}^{L}(X,Y)$ with  $\pi_{\left(\mathfrak{m}^L\left(s;p\right),r\right)}^{L}(S)\leq D$. From Theorem \ref{manaf1} we have $\pi_{\left(\mathfrak{m}^L\left(s;p\right),r\right)}^{L}(S)=D$. \\
\end{proof}

The general inclusion results are the following.

\begin{Proposition}
If $0<p_1\leq p_2$; $0<q_1\leq q_2$; $0<s_1\leq s_2$; $q_j\leq s_j$; $q_j\leq p_j$; $j=1,2$ and $$\frac{1}{q_1}-\frac{1}{q_2}\leq\frac{1}{s_1}-\frac{1}{s_2}\leq\frac{1}{p_1}-\frac{1}{p_2},$$ then $$\Pi_{\left(p_1,\mathfrak{m}^L\left(s_1;q_1\right)\right)}^{L}(X,Y)\subset\Pi_{\left(p_2,\mathfrak{m}^L\left(s_2;q_2\right)\right)}^{L}(X,Y).$$ Moreover $$\pi_{\left(p_2,\mathfrak{m}^L\left(s_2;q_2\right)\right)}^{L}(T)\leq\pi_{\left(p_1,\mathfrak{m}^L\left(s_1;q_1\right)\right)}^{L}(T)$$ for every $T\in\Pi_{\left(p_1,\mathfrak{m}^L\left(s_1;q_1\right)\right)}^{L}(X,Y)$.
\end{Proposition}

\begin{proof}
We consider $\frac{1}{q}=\frac{1}{q_1}-\frac{1}{q_2}$; $\frac{1}{s}=\frac{1}{s_1}-\frac{1}{s_2}$ and $\frac{1}{p}=\frac{1}{p_1}-\frac{1}{p_2}$. We have $p\leq s$. Let $T$ be an arbitrary operator in $\Pi_{\left(p_1,\mathfrak{m}^L\left(s_1;q_1\right)\right)}^{L}(X,Y)$ and $\alpha\subset\mathbb{R}$. Hence
\begin{align}
\left\|(\alpha\cdot\sigma,Tx',Tx'')\Big|\ell_{p_1}\right\|&\leq\pi_{\left(p_1,\mathfrak{m}^L\left(s_1;q_1\right)\right)}^{L}(T)\cdot \mathfrak{m}_{(s_1;q_1)}^{L}(\alpha\cdot\sigma,x',x'') \nonumber \\
&\leq\pi_{\left(p_1,\mathfrak{m}^L\left(s_1;q_1\right)\right)}^{L}(T)\cdot\left\|\lambda\Big|\ell_{s'_1\left(q_1\right)}\right\|\cdot\left\|\left(\frac{\alpha\cdot\sigma}{\lambda},x',x''\right)\Big|\ell_{s_1}^{L,w}\right\|. \nonumber
\end{align}
By using the H\"older inequality and $s'_2(q_2)\leq s'_1(q_1)$, we have
\begin{equation}\label{nine}
\left\|(\alpha\cdot\sigma,Tx',Tx'')\Big|\ell_{p_1}\right\|\leq\pi_{\left(p_1,\mathfrak{m}^L\left(s_1;q_1\right)\right)}^{L}(T)\cdot\left\|\lambda\Big|\ell_{s'_2\left(q_2\right)}\right\|\cdot\left\|\alpha\Big|\ell_{s}\right\|\cdot\left\|\left(\frac{\sigma}{\lambda},x',x''\right)\Big|\ell_{s_2}^{L,w}\right\|. 
\end{equation}
Since $p\leq s$ and taking the infimum over $\lambda\in\ell_{s'_2\left(q_2\right)}$ on the right side of (\ref{nine}), we obtain
$$\left\|(\alpha\cdot\sigma,Tx',Tx'')\Big|\ell_{p_1}\right\|\leq\pi_{\left(p_1,\mathfrak{m}^L\left(s_1;q_1\right)\right)}^{L}(T)\cdot\left\|\alpha\Big|\ell_{p}\right\|\cdot\mathfrak{m}_{(s_2;q_2)}^{L}(\sigma,x',x'').$$
Then
$$\left\|(\sigma,Tx',Tx'')\Big|\ell_{p_2}\right\|\leq\pi_{\left(p_1,\mathfrak{m}^L\left(s_1;q_1\right)\right)}^{L}(T)\cdot\mathfrak{m}_{(s_2;q_2)}^{L}(\sigma,x',x'').$$ Moreover, it follows from Definition \ref{eight} that
$$\pi_{\left(p_2,\mathfrak{m}^L\left(s_2;q_2\right)\right)}^{L}(T)\leq\pi_{\left(p_1,\mathfrak{m}^L\left(s_1;q_1\right)\right)}^{L}(T).$$ 
\end{proof}

 In the case that $q_j=s_j$, for $j=1,2$, this result gives

\begin{Corollary}\label{zoo}
If $0<p_1\leq p_2$; $0<q_1\leq q_2$; $q_j\leq p_j$, $j=1,2$ and $\frac{1}{q_1}-\frac{1}{q_2}\leq\frac{1}{p_1}-\frac{1}{p_2}$, then
$$\Pi_{\left(p_1,\mathfrak{m}^L\left(q_1;q_1\right)\right)}^{L}(X,Y)\subset\Pi_{\left(p_2,\mathfrak{m}^L\left(q_2;q_2\right)\right)}^{L}(X,Y).$$ Moreover $$\pi_{\left(p_2,\mathfrak{m}^L\left(q_2;q_2\right)\right)}^{L}(T)\leq\pi_{\left(p_1,\mathfrak{m}^L\left(q_1;q_1\right)\right)}^{L}(T)$$ for every $T\in\Pi_{\left(p_1,\mathfrak{m}^L\left(q_1;q_1\right)\right)}^{L}(X,Y)$.
\end{Corollary}
 
\begin{proof}
We define $\frac{1}{q}=\frac{1}{q_1}-\frac{1}{q_2}$ and $\frac{1}{p}=\frac{1}{p_1}-\frac{1}{p_2}$. Then we have $p\leq q$. Let $T$ be an arbitrary operator in $\Pi_{\left(p_1,\mathfrak{m}^L\left(q_1;q_1\right)\right)}^{L}(X,Y)$; $\alpha\subset\mathbb{R}$. We have
\begin{align}
\left\|(\alpha\cdot\sigma,Tx',Tx'')\Big|\ell_{p_1}\right\|&\leq\pi_{\left(p_1,\mathfrak{m}^L\left(q_1;q_1\right)\right)}^{L}(T)\cdot\mathfrak{m}_{(q_1;q_1)}^{L}(\alpha\cdot\sigma,x',x'') \nonumber \\
&=\pi_{\left(p_1,\mathfrak{m}^L\left(q_1;q_1\right)\right)}^{L}(T)\cdot\left\|(\alpha\cdot\sigma,x',x'')\Big|\ell_{q_1}^{L,w}\right\|. \nonumber 
\end{align}
By using the H\"older inequality, we obtain
\begin{align}
\left\|(\alpha\cdot\sigma,Tx',Tx'')\Big|\ell_{p_1}\right\|&\leq\pi_{\left(p_1,\mathfrak{m}^L\left(q_1;q_1\right)\right)}^{L}(T) \cdot\left\|\alpha\Big|\ell_{q}\right\|\cdot\left\|(\sigma,x',x'')\Big|\ell_{q_2}^{L,w}\right\| \nonumber \\
&\leq\pi_{\left(p_1,\mathfrak{m}^L\left(q_1;q_1\right)\right)}^{L}(T)\cdot\left\|\alpha\Big|\ell_{p}\right\|\cdot\left\|(\sigma,x',x'')\Big|\ell_{q_2}^{L,w}\right\|. \nonumber
\end{align}
Then
$$\left\|(\sigma,Tx',Tx'')\Big|\ell_{p_2}\right\|\leq\pi_{\left(p_1,\mathfrak{m}^L\left(q_1;q_1\right)\right)}^{L}(T)\cdot\mathfrak{m}_{(q_2;q_2)}^{L}(\sigma,x',x'').$$ Moreover, it follows from Definition \ref{eight} that
$$\pi_{\left(p_2,\mathfrak{m}^L\left(q_2;q_2\right)\right)}^{L}(T)\leq\pi_{\left(p_1,\mathfrak{m}^L\left(q_1;q_1\right)\right)}^{L}(T).$$ 
\end{proof}

\begin{Proposition}\label{see}
If $1\leq s_2\leq s_1\leq\infty$, $0<p_j\leq s_j\leq\infty$, $0<p_1\leq p_2$, $0< r_1\leq r_2$, $r_j\leq p_j$, $j=1,2$ and $\frac{1}{r_1}-\frac{1}{r_2}\leq\frac{1}{p_1}-\frac{1}{p_2}$, then $$\Pi_{\left(\mathfrak{m}^L\left(s_1;p_1\right),r_1\right)}^{L}(X,Y)\subset\Pi_{\left(\mathfrak{m}^L\left(s_2;p_2\right),r_2\right)}^{L}(X,Y).$$
Moreover $$\pi_{\left(\mathfrak{m}^L\left(s_2;p_2\right),r_2\right)}^{L}(S)\leq\pi_{\left(\mathfrak{m}^L\left(s_1;p_1\right),r_1\right)}^{L}(S)$$
for every $S\in\Pi_{\left(\mathfrak{m}^L\left(s_1;p_1\right),r_1\right)}^{L}(X,Y).$
\end{Proposition}
\begin{proof}
Let $S$ be an arbitrary operator in $\Pi_{\left(\mathfrak{m}^L\left(s_1;p_1\right),r_1\right)}^{L}(X,Y)$ and let $T$ be an arbitrary operator in $\Pi_{\left(s_2,\mathfrak{m}^L\left(s_2;s_2\right)\right)}^{L}(Y,Z)$. Then from Corollary \ref{zoo}, we have $T\in\Pi_{\left(s_1,\mathfrak{m}^L\left(s_1;s_1\right)\right)}^{L}(Y,Z)$ with $$\pi_{\left(s_1,\mathfrak{m}^L\left(s_1;s_1\right)\right)}^{L}(T)\leq\pi_{\left(s_2,\mathfrak{m}^L\left(s_2;s_2\right)\right)}^{L}(T).$$ Hence from Theorem \ref{manaf1}, we obtain $T\circ S\in\Pi_{\left(p_1,\mathfrak{m}^L\left(r_1;r_1\right)\right)}^{L}(X,F)$ with
\begin{align}
\pi_{\left(p_1,\mathfrak{m}^L\left(r_1;r_1\right)\right)}^{L}(T\circ S)&\leq\pi_{\left(s_1,\mathfrak{m}^L\left(s_1;s_1\right)\right)}^{L}(T)\cdot\pi_{\left(\mathfrak{m}^L\left(s_1;p_1\right),r_1\right)}^{L}(S)\nonumber \\
&\leq\pi_{\left(s_2,\mathfrak{m}^L\left(s_2;s_2\right)\right)}^{L}(T)\cdot\pi_{\left(\mathfrak{m}^L\left(s_1;p_1\right),r_1\right)}^{L}(S).\nonumber 
\end{align}
From Corollary \ref{zoo}, we get $T\circ S\in\Pi_{\left(p_2,\mathfrak{m}^L\left(r_2;r_2\right)\right)}^{L}(X,F)$ with $$\pi_{\left(p_2,\mathfrak{m}^L\left(r_2;r_2\right)\right)}^{L}(T\circ S)\leq\pi_{\left(p_1,\mathfrak{m}^L\left(r_1;r_1\right)\right)}^{L}(T\circ S).$$
Now for $(g_k)_{k=1}^{n}\subset Y^{\#}$ and $n\in\mathbb{N}$. We define $T\in\Pi_{\left(p_2,\mathfrak{m}^L\left(s_2;p_2\right)\right)}^{L}(Y,\ell_{s_2})$ by the rule $$T(y)=\Big(\left\langle g_{1},y\right\rangle_{(Y^{\#},Y)},\cdots,\left\langle g_{n},y\right\rangle_{(Y^{\#},Y)},0,0,0,\cdots\Big)$$
with $$\pi_{\left(p_2,\mathfrak{m}^L\left(s_2;p_2\right)\right)}^{L}(T)\leq\left\|(g_k)_{k=1}^{n}\Big|\ell_{s_2}(Y^{\#})\right\|.$$
It follows from Theorem \ref{d} that $$\pi_{\left(s_2,\mathfrak{m}^L\left(s_2;s_2\right)\right)}^{L}(T)\leq\left\|(g_k)_{k=1}^{n}\Big|\ell_{s_2}(Y^{\#})\right\|.$$
Then from (\ref{moon}) we have
\begin{align}
\Bigg[\sum\limits_{j=1}^{m}\left|\sigma_j\right|^{p_2}\bigg[\sum\limits_{k=1}^{n}&\left|\left\langle g_k,Sx'_j\right\rangle_{(Y^{\#},Y)}-\left\langle g_k,Sx''_j\right\rangle_{(Y^{\#},Y)}\right|^{s_2}\bigg]^\frac{p_2}{s_2}\Bigg]^\frac{1}{p_2}\nonumber \\
&\leq\pi_{\left(p_2,\mathfrak{m}^L\left(r_2;r_2\right)\right)}^{L}(T\circ S)\cdot\left\|(\sigma,x',x'')\Big|\ell_{r_2}^{L,w}\right\|\nonumber \\
&\leq\pi_{\left(\mathfrak{m}^L\left(s_1;p_1\right),r_1\right)}^{L}(S)\cdot\left\|(g_k)_{k=1}^{n}\Big|\ell_{s_2}(Y^{\#})\right\|\cdot\left\|(\sigma,x',x'')\Big|\ell_{r_2}^{L,w}\right\|. \nonumber
\end{align}
From Theorem \ref{a}, we obtain $S\in\Pi_{\left(\mathfrak{m}^L\left(s_2;p_2\right),r_2\right)}^{L}(X,Y)$ with $$\pi_{\left(\mathfrak{m}^L\left(s_2;p_2\right),r_2\right)}^{L}(S)\leq\pi_{\left(\mathfrak{m}^L\left(s_1;p_1\right),r_1\right)}^{L}(S).$$  
\end{proof}

In the case that $p_j=r_j$, for $j=1, 2$, this result gives

\begin{Corollary}
If $1\leq s_2\leq s_1\leq\infty$, $0<p_j\leq s_j\leq\infty$ and $0<p_1\leq p_2$, then $$\Pi_{\left(\mathfrak{m}^L\left(s_1;p_1\right),p_1\right)}^{L}(X,Y)\subset\Pi_{\left(\mathfrak{m}^L\left(s_2;p_2\right),p_2\right)}^{L}(X,Y).$$
Moreover $$\pi_{\left(\mathfrak{m}^L\left(s_2;p_2\right),p_2\right)}^{L}(S)\leq\pi_{\left(\mathfrak{m}^L\left(s_1;p_1\right),p_1\right)}^{L}(S)$$
for every $S\in\Pi_{\left(\mathfrak{m}^L\left(s_1;p_1\right),p_1\right)}^{L}(X,Y)$.
\end{Corollary}

\section{Chevet$-$Saphar spaces for Lipschitz $\left(\mathfrak{m}^L\left(s;p\right),r\right)-$summing maps}

In this section, the letters $s$, $p$, $r$ will designate elements of $[1,\infty]$; $s'$, $p'$ and $r'$ denote the exponent conjugate to $s$, $p$ and $r$, respectively. 

We start by recalling the definitions, basic properties and theorems of the Chevet$-$Saphar spaces in \cite{CD11}.

 An $E-$valued molecule on $X$ is a finitely supported function $\textbf{m}$ from $X$ into $E$ such that $\sum\limits_{x\in X} \textbf{m}(x)=0$. The vector space of all $E-$valued molecules on $X$ is denoted by $\mathcal{M}(X,E)$. 

Given $x_{1}$, $x_{2}\in X$, define $\textbf{m}_{x_{1}x_{2}}=\chi_{x_{1}}-\chi_{x_{2}}$, where $\chi_{x_{i}}$ stands for the characteristic function on $X$, $i=1,2$. The simplest nonzero molecules, i.e. those of the form $v\textbf{m}_{x_{1}x_{2}}$, for some $x_{1}, x_{2}\in X$ and $v\in E$, are called atoms. Note that any molecule may be expressed (in a non unique way) as a finite sum of atoms. 
 
 Similarly to \cite{CD11}, we define the $p-$th Chevet$-$Saphar norm $cs_{p}$ of a molecule $\textbf{m}$
$$cs_{p}(\textbf{m})=\inf\left\|\sigma\cdot\left\|v\right\|\Big|\ell_{p}\right\|\left\|(\sigma^{-1},x',x'')\Big|\ell_{p^{'}}^{L,w}\right\|$$ where the infimum is taken over all representations $\textbf{m}=\sum\limits_{j=1}^{m} v_{j}\textbf{m}_{x'_j x''_j}$ and $\sigma\subset\mathbb{R}$. 

  Similarly to \cite[Theorem 4.1]{CD11}, the vector space of $E-$valued molecules on $X$, endowed with the norm $cs_p(\cdot)$, forms a normed space denoted by $\mathcal CS_p(X,E)$. 
  
  There is a canonical way of inducing a pairing between $E-$valued molecules on $X$ and functions from $X$ to $E^*$. Given $\textbf{m}\in\mathcal{M}(X,E)$ and a function $T$ from $X$ into $E^*$, this pairing is defined by the rule $$\left\langle T,\textbf{m}\right\rangle=\sum\limits_{x\in X}\left\langle T(x),\textbf{m}(x)\right\rangle.$$ If we know an expression of the molecule as a sum of atoms, say $\textbf{m}=\sum\limits_{j=1}^{m} v_{j}\textbf{m}_{x'_j x''_j}$, then
\begin{align}
\left\langle T,\textbf{m}\right\rangle=\sum\limits_{j=1}^{m}\left\langle Tx'_j-Tx''_j,v_j\right\rangle. \label{ten}
\end{align} 

Also, similarly to \cite [Theorem 4.3]{CD11}, the dual space of $\mathcal CS_p(X,E)$ is canonically identified with the space of Lipschitz $(p',\mathfrak{m}^L(p';p'))-$summing operators from $X$ into $E^*$ by the pairing formula defined in (\ref{ten}). 

For an arbitrary molecule $\textbf{m}\in\mathcal{M}(X,E)$, let us define $$cs_{p',\,r}(\textbf{m})=\inf\left\|\sigma\cdot\left\|v\right\|\Big|\ell_{p'}\right\|\left\|(\sigma^{-1},x',x'')\Big|\ell_{r}^{L,w}\right\|$$ where the infimum is taken over all representations of $\textbf{m}=\sum\limits_{j=1}^{m} v_{j} \textbf{m}_{x'_j x''_j}$ and $\sigma\subset\mathbb{R}$. 
  
 Also, observe that for any Banach space $E$ a Lipschitz map $T$ from $X$ into $Y$ naturally induces a well$-$defined linear map $T_{E}$ from $\mathcal{M}(X,E)$ into $\mathcal{M}(Y,E)$ given by 
$$T_{E}\left(\sum\limits_{j=1}^{m} v_{j}\textbf{m}_{x'_j x''_j}\right)=\sum\limits_{j=1}^{m} v_{j}\textbf{m}_{Tx'_j Tx''_j}.$$ 
 
 Recall that for $0<\beta\leq 1$, a non$-$negative positively homogeneous functional $\rho$ defined on a vector space $U$ is called a $\beta-$seminorm if $\rho(u_1+u_2)^{\beta}\leq\rho(u_1)^{\beta}+\rho(u_2)^{\beta}$ for all $u_1, u_2\in U$. If in addition $\rho$ vanishes only at $0$, it is called a $\beta-$norm.%

\begin{Remark}
\begin{itemize}
\item Recall the definition of the norm $\mu_{p\,,r\,,s}(\cdot)$ in \cite[Sec. 5.1]{CD11}. For the special case $s=\infty$, it is obvious	that $\mu_{p,\,p',\,\infty}(\cdot)=cs_{p}(\cdot)$ and $\mu_{p',\,r,\,\infty}(\cdot)=cs_{p',\,r}(\cdot)$.
\item Recall the definition of Lipschitz $(p,r,s)-$summing maps in \cite [Sec 5.2]{CD11}. For the special case $s=\infty$, we have	$$\Pi_{p,\,r,\,\infty}^{L}(X,E^{\ast})=\Pi_{\left(p,\mathfrak{m}^L\left(r;r\right)\right)}^{L}(X,E^{\ast}).$$
\end{itemize}
\end{Remark}
The next Lemma is a special case of \cite[Theorem 5.1]{CD11}.  

\begin{Lemma}
If $\frac{1}{\beta}=\frac{1}{p'}+\frac{1}{r}\geq 1$, then $\left(\mathcal{M}(X,E),cs_{p',\,r}(\cdot)\right)$ is a $\beta-$normed space.
\end{Lemma}

\begin{Remark}
The $\beta-$normed space $\left(\mathcal{M}(X,E),cs_{p',\,r}(\cdot)\right)$ will be denoted by $\mathcal{CS}_{p',\,r}(X,E)$.
\end{Remark}

 The next Proposition is a special case of \cite[Theorem 5.2]{CD11}. 

\begin{Proposition}
The spaces $\mathcal{CS}_{p',\,r}(X,E)^{\ast}$ and $\Pi_{\left(p,\mathfrak{m}^L\left(r;r\right)\right)}^{L}(X,E^{\ast})$ are isometrically isomorphic via
the canonical pairing defined in (\ref{ten}). 
\end{Proposition}

The following characterization of Lipschitz $\left(\mathfrak{m}^L\left(s;p\right),r\right)-$summing maps between metric spaces is in terms of ideal norms of associated bounded linear operators between Chevet$-$Saphar spaces.

\begin{thm}
Let $S$ from $X$ into $Y$ be a Lipschitz map. The following are equivalent
\begin{enumerate}
	\item $S$ is a Lipschitz $\left(\mathfrak{m}^L\left(s;p\right),r\right)-$summing map.
	\item For every Banach space $G$ (or only $G=\ell_{s'}$), the operator $$S_G:\mathcal{CS}_{p',\,r}(X,G)\longrightarrow\mathcal CS_{s'}(Y,G)$$ is bounded. In this case $$\pi_{\left(\mathfrak{m}^L\left(s;p\right),r\right)}^{L}(S)=\left\|S_{\ell_{s'}}:\mathcal{CS}_{p',\,r}(X,\ell_{s'})\longrightarrow\mathcal CS_{s'}(Y,\ell_{s'})\right\|\geq\left\|S_G:\mathcal{CS}_{p',\,r}(X,G)\longrightarrow\mathcal CS_{s'}(Y,G)\right\|.$$
\end{enumerate}
\end{thm}

\begin{proof}
First, suppose that $S$ is a Lipschitz $\left(\mathfrak{m}^L\left(s;p\right),r\right)-$summing map. Let $\varphi\in\mathcal CS_{s'}(Y,G)^{\ast}$ with $\left\|\varphi\right\|\leq 1$. Since $\mathcal CS_{s'}(Y,G)^{\ast}\equiv\Pi_{\left(s,\mathfrak{m}^L\left(s;s\right)\right)}^{L}(Y,G^{\ast})$, we can identify $\varphi$ with a map $L_{\varphi}\in\Pi_{\left(s,\mathfrak{m}^L\left(s;s\right)\right)}^{L}(Y,G^{\ast})$ with $\pi_{\left(s,\mathfrak{m}^L\left(s;s\right)\right)}^{L}(L_{\varphi})=\left\|\varphi\right\|\leq 1$. 

Let $\textbf{m}$ be a $G-$valued molecule on $X$, say $\textbf{m}=\sum\limits_{j=1}^{m} v_{j} \textbf{m}_{x'_j x''_j}$ with $x'_j, x''_j\in X$ and $v_j\in G$. Then $$S_{G}\left(\textbf{m}\right)=\sum\limits_{j=1}^{m} v_{j}\textbf{m}_{Sx'_j Sx''_j}.$$

The pairing formula defined in (\ref{ten}), the H\"older inequality and Theorem \ref{manaf1} naturally come together to give us
\begin{align}
\left\langle\varphi,S_{G}\left(\textbf{m}\right)\right\rangle&=\sum\limits_{j=1}^{m}\left\langle L_{\varphi}(Sx'_j)-L_{\varphi}(Sx''_j),v_j\right\rangle \nonumber \\
&=\sum\limits_{j=1}^{m}\left\langle\left(L_{\varphi}\circ S\right) x'_j-\left(L_{\varphi}\circ S\right) x''_j,v_j\right\rangle \nonumber \\
&=\left\langle L_{\varphi}\circ S,\textbf{m}\right\rangle. \nonumber 
\end{align}
Hence
\begin{align}
\left|\left\langle\varphi,S_{G}(\textbf{m})\right\rangle\right|&=\left|\left\langle L_{\varphi}\circ S,\textbf{m}\right\rangle\right| \nonumber \\
&\leq\sum\limits_{j=1}^{m}\left|\left\langle\left(L_{\varphi}\circ S\right) x'_j-\left(L_{\varphi}\circ S\right) x''_j,v_j\right\rangle\right| \nonumber \\
&\leq\sum\limits_{j=1}^{m}\left\|\left(L_{\varphi}\circ S\right) x'_j-\left(L_{\varphi}\circ S\right) x''_j\right\|\left\|v_j\right\| \nonumber \\
&\leq\left[\sum\limits_{j=1}^{m}\left|\frac{1}{\sigma_j}\right|^{p}\left\|\left(L_{\varphi}\circ S\right) x'_j-\left(L_{\varphi}\circ S\right) x''_j\right\|^{p}\right]^{\frac{1}{p}}\cdot\left\|\sigma\cdot\left\|v\right\|\Big|\ell_{p'}\right\| \nonumber \\
&\leq\pi_{\left(p,\mathfrak{m}^L\left(r;r\right)\right)}^{L}(L_{\varphi}\circ S)\cdot\left\|\big(\frac{1}{\sigma},x',x''\big)\Big|\ell_r^{L,w}\right\|\cdot\left\|\sigma\cdot\left\|v\right\|\Big|\ell_{p'}\right\| \nonumber \\
&\leq\pi_{\left(\mathfrak{m}^L\left(s;p\right),r\right)}^{L}(S)\cdot\left\|\big(\frac{1}{\sigma},x',x''\big)\Big|\ell_r^{L,w}\right\|\cdot\left\|\sigma\cdot\left\|v\right\|\Big|\ell_{p'}\right\|. \label{twelve} 
\end{align}
Taking the infimum over all representations of $\textbf{m}$ and $\sigma\subset\mathbb{R}$ on the right side of (\ref{twelve}), we have
\begin{equation}\label{thirty}
\left|\left\langle\varphi,S_{G}\left(\textbf{m}\right)\right\rangle\right|\leq\pi_{\left(\mathfrak{m}^L\left(s;p\right),r\right)}^{L}(S)\cdot cs_{p',\,r}(\textbf{m})
\end{equation} 
Taking the supremum over all such $\varphi$ on the left side of (\ref{thirty}), we have 
$$\sup\limits_{\varphi\in B_{\mathcal CS_{s'}(Y,G)^{\ast}}}\left|\left\langle\varphi,S_{G}\left(\textbf{m}\right)\right\rangle\right|\leq\pi_{\left(\mathfrak{m}^L\left(s;p\right),r\right)}^{L}(S)\cdot cs_{p',\,r}(\textbf{m}).$$
Then $$cs_{s'}(S_{G}\left(\textbf{m}\right))\leq\pi_{\left(\mathfrak{m}^L\left(s;p\right),r\right)}^{L}(S)\cdot cs_{p',\,r}(\textbf{m})$$ and $$\left\|S_{G}\right\|\leq\pi_{\left(\mathfrak{m}^L\left(s;p\right),r\right)}^{L}(S).$$

Conversely, suppose that $S_{\ell_{s'}}:\mathcal{CS}_{p',\,r}(X,\ell_{s'})\longrightarrow\mathcal CS_{s'}(Y,\ell_{s'})$ is a bounded linear operator.
Let $(g_k)_{k=1}^{n}\subset Y^{\#}$ and $n\in\mathbb{N}$. We define $T\in\Pi_{\left(p,\mathfrak{m}^L\left(s;p\right)\right)}^{L}(Y,\ell_s)$ by the rule $$T(y)=\Big(\left\langle g_{1},y\right\rangle_{(Y^{\#},Y)},\cdots,\left\langle g_{n},y\right\rangle_{(Y^{\#},Y)},0,0,0,\cdots\Big)$$
with $$\pi_{\left(p,\mathfrak{m}^L\left(s;p\right)\right)}^{L}(T)\leq\left\|(g_k)_{k=1}^{n}\Big|\ell_s(Y^{\#})\right\|.$$
It follows from Theorem \ref{d} that $$\pi_{\left(s,\mathfrak{m}^L\left(s;s\right)\right)}^{L}(T)\leq\left\|(g_k)_{k=1}^{n}\Big|\ell_s(Y^{\#})\right\|.$$   Assume $\textbf{m}$ is an $\ell_{s'}-$valued molecule on $X$, say $\textbf{m}=\sum\limits_{j=1}^{m} v_{j} \textbf{m}_{x'_j x''_j}$ with $x'_j, x''_j\in X$ and $v_j\in\ell_{s'}$. 

It suffices to show that $T\circ S\in\Pi_{\left(p,\mathfrak{m}^L\left(r;r\right)\right)}^{L}(X,\ell_s)$.
\begin{align}
\left\langle T\circ S,\textbf{m}\right\rangle&=\sum\limits_{j=1}^{m}\left\langle (T\circ S)x'_j-(T\circ S)x''_j,v_j\right\rangle\nonumber \\
&=\sum\limits_{j=1}^{m}\left\langle T(Sx'_j)-T(Sx''_j),v_j\right\rangle\nonumber \\
&=\left\langle T,\sum\limits_{j=1}^{m} v_{j} \textbf{m}_{Sx'_j Sx''_j}\right\rangle\nonumber \\
&=\left\langle T,S_{\ell_{s'}}(\textbf{m})\right\rangle. \nonumber  
\end{align}
The H\"older inequality and the definition of Lipschitz $\left(s,\mathfrak{m}^L\left(s;s\right)\right)-$summing maps naturally come together to give us
\begin{align} 
\left|\left\langle T\circ S,\textbf{m}\right\rangle\right|&=\left|\left\langle T,S_{\ell_{s'}}(\textbf{m})\right\rangle\right| \nonumber \\
&\leq\sum\limits_{j=1}^{m}\left|\left\langle T(Sx'_j)-T(Sx''_j),v_j\right\rangle\right| \nonumber \\
&\leq\sum\limits_{j=1}^{m}\left\|T(Sx'_j)-T(Sx''_j)\Big|\ell_{s}\right\|\left\|v_j\Big|\ell_{s'}\right\| \nonumber \\
&\leq\left[\sum\limits_{j=1}^{m}\left|\frac{1}{\sigma_j}\right|^{s}\left\|T(Sx'_j)-T(Sx''_j)\Big|\ell_s\right\|^{s}\right]^{\frac{1}{s}}\cdot\left\|\sigma\cdot\left\|v\right\|\Big|\ell_{s'}\right\| \nonumber \\
&\leq\left\|(g_k)_{k=1}^{n}\Big|\ell_s(Y^{\#})\right\|\cdot\left\|\big(\frac{1}{\sigma},Sx',Sx''\big)\Big|\ell_s^{L,w}\right\|\cdot\left\|\sigma\cdot\left\|v\right\|\Big|\ell_{s'}\right\|. \label{forty}
\end{align}
Taking the infimum over all representations of $\textbf{m}$ and $\sigma\subset\mathbb{R}$ on the right side of (\ref{forty}) and using the boundedness of $S_{\ell_{s'}}$, we have
\begin{align}
\left|\left\langle T\circ S,\textbf{m}\right\rangle\right|&\leq\left\|(g_k)_{k=1}^{n}\Big|\ell_s(Y^{\#})\right\|\cdot cs_{s'}\left(S_{\ell_{s'}}(\textbf{m})\right) \nonumber \\
&\leq\left\|(g_k)_{k=1}^{n}\Big|\ell_s(Y^{\#})\right\|\cdot\left\|S_{\ell_{s'}}\right\|\cdot cs_{p',\,r}(\textbf{m}). \label{fifty}
\end{align}
Therefore, from the duality between $cs_{p',\,r}(\cdot)$ and  $\pi_{\left(p,\mathfrak{m}^L\left(r;r\right)\right)}^{L}(\cdot)$, after taking the supremum over all molecules $\textbf{m}$ with $cs_{p',\,r}(\textbf{m})\leq 1$ on both sides of (\ref{fifty}), we obtain $$\pi_{\left(p,\mathfrak{m}^L\left(r;r\right)\right)}^{L}(T\circ S)\leq\left\|(g_k)_{k=1}^{n}\Big|\ell_s(Y^{\#})\right\|\cdot\left\|S_{\ell_{s'}}\right\|.$$
Then from (\ref{moon}) we have
\begin{align}
\Bigg[\sum\limits_{j=1}^{m}\left|\sigma_j\right|^{p}\bigg[\sum\limits_{k=1}^{n}&\left|\left\langle g_k,Sx'_j\right\rangle_{(Y^{\#},Y)}-\left\langle g_k,Sx''_j\right\rangle_{(Y^{\#},Y)}\right|^{s}\bigg]^\frac{p}{s}\Bigg]^\frac{1}{p}\nonumber \\
&\leq\left\|S_{\ell_{s'}}\right\|\cdot\left\|(g_k)_{k=1}^{n}\Big|\ell_s(Y^{\#})\right\|\cdot\left\|(\sigma,x',x'')\Big|\ell_r^{L,w}\right\|. \nonumber
\end{align}
By Theorem \ref{a} we get  $S$ is a Lipschitz $\left(\mathfrak{m}^L\left(s;p\right),r\right)-$summing map with $\pi_{\left(\mathfrak{m}^L\left(s;p\right),r\right)}^{L}(S)\leq\left\|S_{\ell_{s'}}\right\|$.\\ 
\end{proof}

\section{APPLICATIONS}
\subsection{An 'interpolation style' theorem}
As it so often happens with many constants associated to mappings, it is not easy to calculate the Lipschitz $\left(\mathfrak{m}^L\left(s;p\right),r\right)-$summing constant of a specific map. The following 'interpolation style' theorem is
based on \cite[Lemma 5]{Puhl77} and gives useful bounds that are sufficient in some cases.

\begin{mytheorem}\label{www}
If $1\leq s\leq \infty$, $0< p\leq s\leq \infty$, $p\geq r$, then every Lipschitz $\left(p,\mathfrak{m}^L\left(r;r\right)\right)-$\\summing operator $S$ from $X$ into $Y$ is Lipschitz $\left(\mathfrak{m}^L\left(s;p\right),r\right)-$summing and satisfies $$\pi_{\left(\mathfrak{m}^L\left(s;p\right),r\right)}^{L}(S)\leq \pi_{\left(p,\mathfrak{m}^L\left(r;r\right)\right)}^{L}(S)^{\frac{p}{s'\left(p\right)}}\cdot Lip(S)^{\frac{p}{s}}.$$
\end{mytheorem}

\begin{proof}
From Corollary \ref{e} and the ideal property of Lipschitz $\left(p,\mathfrak{m}^L\left(r;r\right)\right)-$summing operators we conclude that $S$ is a Lipschitz $\left(\mathfrak{m}^L\left(s;p\right),r\right)-$summing operator. 

We recall that for $0<p\leq s\leq\infty$, let the index $s'\left(p\right)$ is determined by the equation $$\frac{1}{s'\left(p\right)}+\frac{1}{s}=\frac{1}{p}.$$ Now , let $\sigma_1,\cdots,\sigma_m\in\mathbb{R}$; $x'_1,\cdots, x'_m$, $x''_1,\cdots,x''_m\in X$ and $m\in\mathbb{N}$. For any probability measure $\mu$ on $B_{Y^{\#}}$, from the point wise inequality $$\left|\left\langle g,y'\right\rangle_{(Y^{\#},Y)}-\left\langle g,y''\right\rangle_{(Y^{\#},Y)}\right|\leq Lip(g)\cdot d_Y(y',y'')$$ for every $y'$, $y''\in Y$ and $g\in Y^{\#}$, we obtain
\begin{align}
\Bigg[\sum\limits_{j=1}^{m}\bigg[\int\limits_{B_{Y^{\#}}}\left|\sigma_j\right|^{s}&\left|\left\langle g,Tx'_j\right\rangle_{(Y^{\#},Y)}-\left\langle g,Tx''_j\right\rangle_{(Y^{\#},Y)}\right|^{s}d\mu(g)\bigg]^\frac{p}{s}\Bigg]^{\frac{1}{p}} \nonumber \\
&\leq\Bigg[\sum\limits_{j=1}^{m}\bigg(\int\limits_{B_{Y^{\#}}}\left|\sigma_j\right|^{p}\left|\left\langle g,Tx'_j\right\rangle_{(Y^{\#},Y)}-\left\langle g,Tx''_j\right\rangle_{(Y^{\#},Y)}\right|^{p}d\mu(g)\bigg)^\frac{p}{s}\nonumber \\
&\cdot\bigg(\left|\sigma_j\right|^{\frac{(s-p)\cdot p}{s}}\cdot d_Y(Sx'_j,Sx''_j)^{^{\frac{(s-p)\cdot p}{s}}}\bigg)\Bigg]^{\frac{1}{p}}.\nonumber 
\end{align}
Noting that $p=\frac{(s-p)\cdot s'\left(p\right)}{s}$ and using the H\"older inequality, we have

\begin{align}\label{hous1}
\Bigg[\sum\limits_{j=1}^{m}\bigg[\int\limits_{B_{Y^{\#}}}&\left|\sigma_j\right|^{s}\left|\left\langle g,Tx'_j\right\rangle_{(Y^{\#},Y)}-\left\langle g,Tx''_j\right\rangle_{(Y^{\#},Y)}\right|^{s}d\mu(g)\bigg]^\frac{p}{s}\Bigg]^{\frac{1}{p}} \nonumber \\
&\leq\Bigg[\sum\limits_{j=1}^{m}\int\limits_{B_{Y^{\#}}}\left|\sigma_j\right|^{p}\left|\left\langle g,Tx'_j\right\rangle_{(Y^{\#},Y)}-\left\langle g,Tx''_j\right\rangle_{(Y^{\#},Y)}\right|^{p}d\mu(g)\Bigg]^\frac{1}{s} \nonumber \\
&\cdot\Bigg[\sum\limits_{j=1}^{m}\left|\sigma_j\right|^{p}\cdot d_Y(Sx'_j,Sx''_j)^{p}\Bigg]^{\frac{1}{s'\left(p\right)}}. 
\end{align}
On the one hand, the fact that $S$ is a Lipschitz $\left(p,\mathfrak{m}^L\left(r;r\right)\right)-$summing means that 
\begin{equation}\label{hous2}
\Bigg[\sum\limits_{j=1}^{m}\left|\sigma_j\right|^{p}\cdot d_Y(Sx'_j,Sx''_j)^{p}\Bigg]^{\frac{1}{s'\left(p\right)}}\leq\pi_{\left(p,\mathfrak{m}^L\left(r;r\right)\right)}^{L}(S)^{\frac{p}{s'\left(p\right)}}\cdot\sup\limits_{f\in B_{{X}^{\#}}}\Bigg[\bigg(\sum\limits_{j=1}^{m}\left|\sigma_j\right|^{r}\left|fx'_j-fx''_j\right|^{r}\bigg)^{\frac{p}{r}}\Bigg]^\frac{1}{s'\left(p\right)}.
\end{equation}
On the other hand, we have
\begin{align}\label{hous3}
\Bigg[\sum\limits_{j=1}^{m}\int\limits_{B_{Y^{\#}}}\left|\sigma_j\right|^{p}&\left|\left\langle g,Tx'_j\right\rangle_{(Y^{\#},Y)}-\left\langle g,Tx''_j\right\rangle_{(Y^{\#},Y)}\right|^{p}d\mu(g)\Bigg]^\frac{1}{s}\nonumber \\
&\leq Lip(S)^\frac{p}{s}\cdot\sup\limits_{f\in B_{{X}^{\#}}}\Bigg[\bigg(\sum\limits_{j=1}^{m}\left|\sigma_j\right|^{r}\left|fx'_j-fx''_j\right|^{r}\bigg)^{\frac{p}{r}}\Bigg]^\frac{1}{s}.
\end{align}
Putting (\ref{hous1}), (\ref{hous2}) and (\ref{hous3}) together gives
\begin{align}
\Bigg[\sum\limits_{j=1}^{m}\bigg[\int\limits_{B_{Y^{\#}}}\left|\sigma_j\right|^{s}&\left|\left\langle g,Tx'_j\right\rangle_{(Y^{\#},Y)}-\left\langle g,Tx''_j\right\rangle_{(Y^{\#},Y)}\right|^{s} d\mu(g)\bigg]^\frac{p}{s}\Bigg]^{\frac{1}{p}} \nonumber \\
&\leq\pi_{\left(p,\mathfrak{m}^L\left(r;r\right)\right)}^{L}(S)^{\frac{p}{s'\left(p\right)}}\cdot Lip(S)^\frac{p}{s}\cdot\left\|(\sigma,x',x'')\Big|\ell_r^{L,w}\right\| \nonumber 
\end{align}
and thus the required conclusion follows from Theorem \ref{a}.
\end{proof}

If we combine Theorem \ref{www} and Corollary \ref{zoo}, then we have the following result.

\begin{myCorollary}\label{aaa}
If $1\leq s\leq\infty$, $0< p\leq s\leq\infty$, $p\geq r$, then every Lipschitz $\left(r,\mathfrak{m}^L\left(r;r\right)\right)-$\\summing map $S$ from $X$ into $Y$ is Lipschitz $\left(\mathfrak{m}^L\left(s;p\right),r\right)-$summing and satisfies $$\pi_{\left(\mathfrak{m}^L\left(s;p\right),r\right)}^{L}(S)\leq \pi_{\left(r,\mathfrak{m}^L\left(r;r\right)\right)}^{L}(S)^{\frac{p}{s'\left(p\right)}}\cdot Lip(S)^{\frac{p}{s}}.$$
\end{myCorollary}

\subsection{The identity on a finite discrete metric space.}
Let $D_n$ stand for the discrete metric space on $n$ points. Assume $1\leq s\leq \infty$, $0< p\leq s\leq \infty$, $p\geq r$. 

J. D. Farmer and W. B. Johnson \cite{J09} proved that 
$$\pi_{\left(r,\mathfrak{m}^L\left(r;r\right)\right)}^{L}(I_{D_n})=\big(2-\frac{2}{n}\big)^\frac{1}{r}$$
for any $r\in[1,\infty]$ and from Corollary \ref{aaa}, we obtain
\begin{align}
\pi_{\left(\mathfrak{m}^L\left(s;p\right),r\right)}^{L}(I_{D_n})&\leq\pi_{\left(r,\mathfrak{m}^L\left(r;r\right)\right)}^{L}(I_{D_n})^{\frac{p}{s'\left(p\right)}}\cdot 1  \nonumber \\
&=\Big[2-\frac{2}{n}\Big]^{\frac{p}{r\cdot s'\left(p\right)}}.\nonumber
\end{align}
In the case that $p=r$, this result gives $\pi_{\left(\mathfrak{m}^L\left(s;p\right),p\right)}^{L}(I_{D_n})=\Big[2-\frac{2}{n}\Big]^{\frac{1}{p}-\frac{1}{s}}$ , proved in \cite{JA12}.

\subsection{The general 'interpolation style' theorem}
Corollary \ref{aaa} is in fact a particular case of the following more general theorem.

\begin{mytheorem}\label{f}
Let $0<\theta<1$; $0<p\leq s, s_{0}, s_{1}\leq\infty$ and $r\leq p$. Define $\frac{1}{s}=\frac{1-\theta}{s_{0}}+\frac{\theta}{s_{1}}$. For Lipschitz map $S$ from $X$ into $Y$, $$\pi_{\left(\mathfrak{m}^L\left(s;p\right),r\right)}^{L}(S)\leq\pi_{\left(\mathfrak{m}^L\left(s_{0};p\right),r\right)}^{L}(S)^{1-\theta}\cdot\pi_{\left(\mathfrak{m}^L\left(s_{1};p\right),r\right)}^{L}(S)^{\theta}.$$
\end{mytheorem}

\begin{proof}
The $p-$conjugate $s'\left(p\right)$; $s_{0}'\left(p\right)$ and $s_{1}'\left(p\right)$  are determined by the following equations
\begin{equation}\nonumber
\frac{1}{s'\left(p\right)}+\frac{1}{s}=\frac{1}{p}; \ \frac{1}{s_{0}'\left(p\right)}+\frac{1}{s_{0}}=\frac{1}{p}\  and \ \frac{1}{s_{1}'\left(p\right)}+\frac{1}{s_{1}}=\frac{1}{p},
\end{equation}
respectively. Note that $\frac{1}{s'\left(p\right)}=\frac{1-\theta}{s_{0}'\left(p\right)}+\frac{\theta}{s_{1}'\left(p\right)}$. 

Let $\sigma_1,\cdot\cdot\cdot,\sigma_m\in\mathbb{R}$; $x'_1,\cdot\cdot\cdot,x'_m$, $x''_1,\cdot\cdot\cdot,x''_m\in X$ and $m\in\mathbb{N}$. Given $\epsilon > 0$, from (\ref{three}), for each $i=0, 1$ there exists a sequence $\tau_{i}\in\ell_{s_{i}'\left(p\right)}$ such that 
\begin{align}
\left\|\tau_{i}\Big|\ell_{s_{i}'\left(p\right)}\right\|\left\|\big(\frac{\sigma}{\tau_{i}},Sx',Sx''\big)\Big|\ell_{s_{i}}^{L,w}\right\|&\leq (1+\epsilon)\cdot\mathfrak{m}_{(s_{i};p)}^{L}(\sigma,Sx',Sx'') \nonumber \\
&\leq (1+\epsilon)\cdot\pi_{\left(\mathfrak{m}^L\left(s_{i};p\right),r\right)}^{L}(S)\cdot\left\|(\sigma,x',x'')\Big|\ell_r^{L,w}\right\|. \nonumber 
\end{align}
Moreover, dividing by appropriate constant we may assume that in fact
$$\left\|\tau_{i}\Big|\ell_{s_{i}'\left(p\right)}\right\|\leq (1+\epsilon)\cdot\pi_{\left(\mathfrak{m}^L\left(s_{i};p\right),r\right)}^{L}(S)\cdot\left\|(\sigma,x',x'')\Big|\ell_r^{L,w}\right\|$$ and $$\left\|\big(\frac{\sigma}{\tau_{i}},Sx',Sx''\big)\Big|\ell_{s_{i}}^{L,w}\right\|\leq 1.$$
For $1\leq j\leq m$, set $\left|\tau_j\right|=\left|\tau_{j,0}\right|^{1-\theta}\cdot\left|\tau_{j,1}\right|^{\theta}$. Then by the H\"older inequality, we have
\begin{align}
\left\|\tau\Big|\ell_{s'\left(p\right)}\right\|&\leq\left\|\tau_0\Big|\ell_{s_{0}'\left(p\right)}\right\|^{1-\theta}\cdot\left\|\tau_1\Big|\ell_{s_{1}'\left(p\right)}\right\|^{\theta} \nonumber \\
&\leq (1+\epsilon)\cdot\pi_{\left(\mathfrak{m}^L\left(s_{0};p\right),r\right)}^{L}(S)^{1-\theta}\cdot\pi_{\left(\mathfrak{m}^L\left(s_{1};p\right),r\right)}^{L}(S)^{\theta}\cdot\left\|(\sigma,x',x'')\Big|\ell_r^{L,w}\right\| \nonumber
\end{align}
On the other hand, it follows from 
$$\left|\frac{\sigma_j}{\tau_j}\right|\cdot\left|fx'_j-fx''_j\right|=\frac{\left|\sigma_j\right|^{1-\theta}}{\left|\tau_{j,0}\right|^{1-\theta}}\cdot\left|fx'_j-fx''_j\right|^{1-\theta}\cdot\frac{\left|\sigma_j\right|^{\theta}}{\left|\tau_{j,1}\right|^{\theta}}\cdot\left|fx'_j-fx''_j\right|^{\theta}$$ that
$$\left\|\big(\frac{\sigma}{\tau},Sx',Sx''\big)\Big|\ell_{s}^{L,w}\right\|\leq\prod\limits_{i=0}^{1}\left\|\big(\frac{\sigma}{\tau_{i}},Sx',Sx''\big)\Big|\ell_{s_{i}}^{L,w}\right\|\leq 1.$$
Then 
\begin{align}
\mathfrak{m}_{(s;p)}^{L}(\sigma,Sx',Sx'')&\leq\left\|\tau\Big|\ell_{s'\left(p\right)}\right\|\cdot\left\|\big(\frac{\sigma}{\tau},Sx',Sx''\big)\Big|\ell_{s}^{L,w}\right\|\nonumber \\
&\leq (1+\epsilon)\cdot\pi_{\left(\mathfrak{m}^L\left(s_{0};p\right),r\right)}^{L}(S)^{1-\theta}\cdot\pi_{\left(\mathfrak{m}^L\left(s_{1};p\right),r\right)}^{L}(S)^{\theta}\cdot\left\|(\sigma,x',x'')\Big|\ell_r^{L,w}\right\|.\nonumber
\end{align}
Hence $S$ is a Lipschitz $\left(\mathfrak{m}^L\left(s;p\right),r\right)-$summing map with $$\pi_{\left(\mathfrak{m}^L\left(s;p\right),r\right)}^{L}(S)\leq (1+\epsilon)\cdot\pi_{\left(\mathfrak{m}^L\left(s_{0};p\right),r\right)}^{L}(S)^{1-\theta}\cdot\pi_{\left(\mathfrak{m}^L\left(s_{1};p\right),r\right)}^{L}(S)^{\theta}.$$
By letting $\epsilon\longrightarrow 0^{+}$, our result is proved.\\
\end{proof}

Let $0<p\leq s\leq\infty$ and $r\leq p$. We say that a metric space $X$ is an $((s;p),r)-$space if the identity map on $X$ is Lipschitz $\left(\mathfrak{m}^L\left(s;p\right),r\right)-$summing. The following corollary shows that the class of $((s;p),r)-$spaces does not depend on $p$ and $r$.

\begin{myCorollary}
If $0<p\leq s\leq\infty$, $r_{0}\leq r_{1}$ and $r_{1}\leq p$, then $X$ is an $((s;r_{0}),r_{0})-$space if and only if it is an $((s;p),r_{1})-$space. Moreover $$\pi_{\left(\mathfrak{m}^L\left(s;p\right),r_{1}\right)}^{L}(I_{X})\leq\pi_{\left(\mathfrak{m}^L\left(s;r_{0}\right),r_{0}\right)}^{L}(I_{X})\leq\pi_{\left(\mathfrak{m}^L\left(s;p\right),r_{1}\right)}^{L}(I_{X})^{\frac{1}{\theta}}$$ where $\theta$ is defined by $\frac{1}{r_{1}}=\frac{1-\theta}{s}+\frac{\theta}{r_{0}}$.
\end{myCorollary}

\begin{proof}
Assume $X$ is an $((s;r_{0}),r_{0})-$space. By Proposition \ref{see} we have $$\pi_{\left(\mathfrak{m}^L\left(s;p\right),r_{1}\right)}^{L}(I_{X})\leq\pi_{\left(\mathfrak{m}^L\left(s;r_{0}\right),r_{0}\right)}^{L}(I_{X}).$$

Conversely, suppose that $X$ is an $((s;p),r_{1})-$space. The composition property from Proposition \ref{hous} provides us with the inequality $$\pi_{\left(\mathfrak{m}^L\left(s;r_{0}\right),r_{0}\right)}^{L}(I_{X})\leq\pi_{\left(\mathfrak{m}^L\left(s;p\right),r_{1}\right)}^{L}(I_{X})\cdot\pi_{\left(\mathfrak{m}^L\left(r_{1};r_{0}\right),r_{0}\right)}^{L}(I_{X}).$$
Now from Theorem \ref{f} we have 
\begin{align}
\pi_{\left(\mathfrak{m}^L\left(r_{1};r_{0}\right),r_{0}\right)}^{L}(I_{X})&\leq\pi_{\left(\mathfrak{m}^L\left(s;r_{0}\right),r_{0}\right)}^{L}(I_{X})^{1-\theta}\cdot\pi_{\left(\mathfrak{m}^L\left(r_{0};r_{0}\right),r_{0}\right)}^{L}(I_{X})^{\theta}\nonumber \\
&=\pi_{\left(\mathfrak{m}^L\left(s;r_{0}\right),r_{0}\right)}^{L}(I_{X})^{1-\theta}.\nonumber 
\end{align}
So we obtain 
\begin{equation}\label{tee}
\pi_{\left(\mathfrak{m}^L\left(s;r_{0}\right),r_{0}\right)}^{L}(I_{X})\leq\pi_{\left(\mathfrak{m}^L\left(s;p\right),r_{1}\right)}^{L}(I_{X})\cdot\pi_{\left(\mathfrak{m}^L\left(s;r_{0}\right),r_{0}\right)}^{L}(I_{X})^{1-\theta}.
\end{equation}
This finally leads to $$\pi_{\left(\mathfrak{m}^L\left(s;r_{0}\right),r_{0}\right)}^{L}(I_{X})\leq\pi_{\left(\mathfrak{m}^L\left(s;p\right),r_{1}\right)}^{L}(I_{X})^{\frac{1}{\theta}}.$$ 
\end{proof}

\subsection{Lipschitz $\left(r,\mathfrak{m}^L\left(r;r\right)\right)-$summing maps for $0<r<1$} 
A. Pietsch \cite[Sec. 21.2.11]{P80} proved that all operator ideals $\mathfrak{P}_{r}$ with $0<r<1$ coincide. We need the following interesting result that is useful to prove that the classes of Lipschitz $\left(r,\mathfrak{m}^L\left(r;r\right)\right)-$\\summing maps with $0<r<1$ coincide.

\begin{mytheorem}
If $0<p\leq q<s<1$, then $\Pi_{\left(\mathfrak{m}^L\left(s;q\right),p\right)}^{L}(X,Y)\cap\mathbb{L}_{x_{0}}(X,Y)=\mathbb{L}_{x_{0}}(X,Y)$.
\end{mytheorem}

\begin{proof}
Let $S\in\mathbb{L}_{x_{0}}(X,Y)$ with Lipschitz dual operator $S^{\#}\in\mathfrak{L}(Y^{\#},X^{\#})$. From Theorem \ref{type4} we obtain the operator $S^{\#}$ is of $(s,p)-$type. Hence from Theorem \ref{type10} we get $S$ is a Lipschitz $\left(\mathfrak{m}^L\left(s;q\right),p\right)-$summing map.\\
\end{proof}

In the case that $p=q$, if we combine Theorem \ref{type4} and Corollary \ref{type5}, then we have the following result.

\begin{myCorollary}\label{type6}
If $0<p<s<1$, then $\Pi_{\left(\mathfrak{m}^L\left(s;p\right),p\right)}^{L}(X,Y)\cap\mathbb{L}_{x_{0}}(X,Y)=\mathbb{L}_{x_{0}}(X,Y)$.
\end{myCorollary}

\begin{mytheorem}
If $0<p<s<1$, then $\Pi_{\left(p,\mathfrak{m}^L\left(p;p\right)\right)}^{L}(X,Y)=\Pi_{\left(s,\mathfrak{m}^L\left(s;s\right)\right)}^{L}(X,Y)$.
\end{mytheorem}

\begin{proof}
From Corollary \ref{zoo} we have $\Pi_{\left(p,\mathfrak{m}^L\left(p;p\right)\right)}^{L}(X,Y)\subset\Pi_{\left(s,\mathfrak{m}^L\left(s;s\right)\right)}^{L}(X,Y)$. To show the converse, let $T\in\Pi_{\left(s,\mathfrak{m}^L\left(s;s\right)\right)}^{L}(X,Y)$. From Corollary \ref{type6} we have $I_{X}$ is a Lipschitz $\left(\mathfrak{m}^L\left(s;p\right),p\right)-$summing map. Hence from Corollary \ref{manaf100}, we obtain that $T$ is a Lipschitz $\left(p,\mathfrak{m}^L\left(p;p\right)\right)-$summing map. \\
\end{proof}

\section{Concluding Remarks}
\begin{enumerate}
  \item From Concluding Remarks \ref{66} (point 7) and the inclusion result $$\Pi_{\left(p,\mathfrak{m}^L\left(s;s\right)\right)}^{L}(X,Y)\subset\Pi_{\left(p,\mathfrak{m}^L\left(s;q\right)\right)}^{L}(X,Y),$$ we conclude that $$\Pi_{\left(p,\mathfrak{m}^L\left(s;s\right)\right)}^{L}(X,Y)=\mathbb{L}(X,Y)=\Pi_{\left(p,\mathfrak{m}^L\left(s;q\right)\right)}^{L}(X,Y)$$ for every $1\leq s< p$ and $0<q\leq s$.
	\item We recall that for $0<p\leq s\leq\infty$, the index $s'\left(p\right)$ is determined by the equation 
$$\frac{1}{s'\left(p\right)}+\frac{1}{s}=\frac{1}{p}.$$
In order to simplify our notations we write $\tilde{r}=s'\left(p\right)$. In the linear case if $\frac{1}{\tilde{r}}+\frac{1}{s}=\frac{1}{p}\leq 1$, then $$\left\|\cdot\right\|_{\left(\mathfrak{m}\left(s;p\right),p\right)}\leq\left\|\cdot\right\|_{\left(\tilde{r},\mathfrak{m}\left(\tilde{r};\tilde{r}\right)\right)}$$ for every bounded linear operator in arbitrary Banach spaces, see \cite [Sec. 20]{P80}.

\ \ \ \! Now if we assume this condition is also true in the nonlinear case (Lipschitz) and apply it in Corollary  \ref{manaf100}, we obtain the following result.

\begin{thm}\label{manaf3}
Let $0<p\leq s\leq\infty$ and $\frac{1}{p}=\frac{1}{s}+\frac{1}{\tilde{r}}\leq 1$. If $S\in\Pi_{\left(\mathfrak{m}^L\left(s;p\right),p\right)}^{L}(X,Y)$; $T\in\Pi_{\left(s,\mathfrak{m}^L\left(s;s\right)\right)}^{L}(Y,Z)$ and $\pi_{\left(\mathfrak{m}^L\left(s;p\right),p\right)}^{L}(S)\leq\pi_{\left(\tilde{r},\mathfrak{m}^L\left(\tilde{r};\tilde{r}\right)\right)}^{L}(S)$, then $T\circ S\in\Pi_{\left(p,\mathfrak{m}^L\left(p;p\right)\right)}^{L}(X,Z)$. Moreover
$$\pi_{\left(p,\mathfrak{m}^L\left(p;p\right)\right)}^{L}(T\circ S)\leq\pi_{\left(s,\mathfrak{m}^L\left(s;s\right)\right)}^{L}(T)\cdot\pi_{\left(\tilde{r},\mathfrak{m}^L\left(\tilde{r};\tilde{r}\right)\right)}^{L}(S).$$
\end{thm}

\begin{Remark}
Theorem \ref{manaf3} gives a sufficient condition for a Lipschitz composition formula as in the linear case of A. Pietsch \cite{P67}.
\end{Remark}
 \item As before, we use the abbreviation $\pi_{p}^{L}(\cdot)=\pi_{\left(p,\mathfrak{m}^L\left(p;p\right)\right)}^{L}(\cdot)$. 

\ \ \ \! J. D. Farmer and W. B. Johnson \cite [Problem 1]{J09} asked, if the composition formula
\begin{equation}\label{buch}
\pi_{p}^{L}(T\circ S)\leq\pi_{r}^{L}(T)\cdot\pi_{s}^{L}(S).
\end{equation} 
in Theorem \ref{manaf3} is true for arbitrary Lipschitz $r-$summing maps $T$, Lipschitz $s-$summing maps $S$ and $\frac{1}{p}\leq(\frac{1}{r}+\frac{1}{s})\wedge 1$.

\ \ \ \! In the forthcoming paper \cite{H13}, we will provide an algorithm to compute the $\pi_{p}^{L}-$summing norm of maps between finite metric spaces exactly. With the help of this algorithm, we show that (\ref{buch}) is in general not true, in contrast to the situation for linear operators. Here we just state the example, details will be provided in \cite{H13}.

\begin{exampl}
We use three$-$point metric spaces $X=\left\{x_{0},\,x_{1},\,x_{2}\right\}$; $Y=\left\{y_{0},\,y_{1},\,y_{2}\right\}$ with 
$$d_{X}(x_{0},x_{1})=d_{X}(x_{1},x_{2})=d_{X}(x_{0},x_{2})=d_{Y}(y_{0},y_{1})=1$$ and 
$$d_{Y}(y_{1},y_{2})=d_{Y}(y_{0},y_{2})=2.$$
Let $S$ from $X$ into $Y$ be the map defined by $$Sx_{j}=y_{j};\  j=0,1,2.$$ 

\ \ \ \! Obviously, $S$ is a Lipschitz map with $Lip(S)=2$. Then the algorithm from \cite{H13} can be used to compute 
$$\pi_{1}^{L}(S)=\frac{5}{2},\  \pi_{2}^{L}(S)=\frac{3}{\sqrt{2}}\  and\  \pi_{2}^{L}(I_{Y})=\frac{\sqrt{11}}{2\cdot\sqrt{2}}.$$
Since $$\frac{5}{2}>\frac{3}{\sqrt{2}}\cdot\frac{\sqrt{11}}{2\cdot\sqrt{2}}$$ this is a counterexample to (\ref{buch}) in the case $r=s=2$, $p=1$ and $T=I_{Y}$.
\end{exampl}
\ \ \ \! In \cite{H13}, we elaborate on this example in detail and provide counterexamples for other values of $p$, $r$ and $s$.
We finish with an application of Corollary \ref{aaa} to estimate rather accurately $\left(\mathfrak{m}^L\left(s;p\right),r\right)-$summing norms for the map from the above example.

\begin{exampl}
If $X$, $Y$ and $S$ be defined in the above example, then Corollary \ref{aaa} gives
\begin{align}
\pi_{\left(\mathfrak{m}^L\left(s;p\right),2\right)}^{L}(S)&\leq\Big[\pi_{2}^{L}(S)\Big]^{\frac{p}{s'\left(p\right)}}\cdot Lip(S)^{\frac{p}{s}}\nonumber \\
&\leq\Big[\frac{3}{\sqrt{2}}\Big]^{\frac{p}{s'\left(p\right)}}\cdot 2^{\frac{p}{s}}. \nonumber 
\end{align}

In the special case $s=4$, $p=3$ we e.g. obtain $s'\left(p\right)=12$ and 
$$2=Lip(S)\leq\pi_{\left(\mathfrak{m}^L\left(4;3\right),2\right)}^{L}(S)\leq 2.029663590.$$
\end{exampl}
\end{enumerate}
\textbf{Acknowledgment}
The author thanks Aicke Hinrichs for his advice and encouragement during the preparation of this paper.


\end{document}